\documentclass[a4paper,11pt,reqno]{amsart}

\usepackage{amssymb,latexsym,amsthm,amsmath,amsfonts}
\usepackage{caption,color,graphicx}
\usepackage[numbers,comma,square,sort&compress]{natbib}
\usepackage{hyperref} 
\usepackage[text={5.5in,9.5in},centering]{geometry}
\usepackage{bbm}
\usepackage{enumitem}

\usepackage{mathtools}
\DeclarePairedDelimiter{\ceil}{\lceil}{\rceil}

\setcaptionmargin{0.25in}

\newtheorem{lemma}{Lemma}[section]
\newtheorem{corollary}{Corollary}[section]
\newtheorem{definition}{Definition}[section]

\newtheorem{proposition}{Proposition}[section]
\newtheorem{remark}{Remark}[section]
\newtheorem{theorem}{Theorem}[section]

\title{Random interval diffeomorphisms}
 \author{Masoumeh Gharaei}
 \address{KdV Institute for Mathematics, University of Amsterdam, Science park 107, 1098 XG Amsterdam, Netherlands}
  \email{m.gharaei@uva.nl}
 \author{Ale Jan Homburg}
 \address{KdV Institute for Mathematics, University of Amsterdam, Science park 107, 1098 XG Amsterdam, Netherlands}
 \address{Department of Mathematics, VU University Amsterdam, De Boelelaan 1081, 1081 HV Amsterdam, Netherlands}
 \email{a.j.homburg@uva.nl}

\begin{document}

\begin{abstract}
We consider a class of step skew product systems of interval diffeomorphisms over shift operators,
as a means to study random compositions of interval diffeomorphisms.
The class is chosen to present in a simplified setting intriguing phenomena of intermingled basins, master-slave synchronization and on-off intermittency. We provide a self-contained discussion of these phenomena. 
%
\end{abstract}

\maketitle

\section{Introduction} 

We deal with the dynamics of specific step skew product systems
$F^+:\Sigma_2^+ \times \mathbb{I} \to \Sigma_2^+ \times \mathbb{I}$, where $\Sigma_2^+ = \{ 1,2\}^\mathbb{N}$ and $\mathbb{I} = [0,1]$, of the form
\[
{F^+}(\omega,x) = (\sigma \omega , f_{\omega_0} (x) ).
\] 
Here $\sigma: \Sigma_2^+ \to \Sigma_2^+$ is the shift operator; $(\sigma \omega)_i = \omega_{i+1}$ for $\omega = (\omega_i)_0^\infty$ and $f_1$, $f_2$ are $C^2$ diffeomorphisms on $\mathbb{I}$ 
that fulfill the following conditions:
\begin{enumerate}
\item $f_i (0) = 0$, $f_i(1) = 1$ for $i=1,2$;
\item $f_1(x) < x$ for $x \in (0,1)$;
\item $f_2 (x) > x$ for $x \in (0,1)$.
\end{enumerate}
We review and present a self-contained study of the dynamics of such skew product systems, 
characterizing the different possible dynamics. 
This may seem a restrictive setup, but these systems exhibit a wealth of dynamical behavior that serves as models for dynamics in more general systems. 

These step skew product systems provide a setting to study all possible compositions of the two maps $f_1,f_2$ in a single framework.
Indeed, for initial conditions $(\omega,x) \in \Sigma_2^+ \times \mathbb{I}$, the coordinate in $\mathbb{I}$ iterates as
\begin{equation}\label{e:comp} 
x, \,  f_{\omega_0} (x), \, f_{\omega_1} \circ f_{\omega_0}(x), \, f_{\omega_2}\circ f_{\omega_1} \circ f_{\omega_0} (x), \, ...
\end{equation}
The maps $f_1$, $f_2$ simply move points to either smaller or larger values.
We will pick the diffeomorphisms $f_1$ and $f_2$ randomly, independently at each iterate, 
with positive probabilities $p_1$ and $p_2 = 1 - p_1$.
This corresponds to taking a Bernoulli measure on $\Sigma_2^+$ from which we pick $\omega$. 
The obtained random compositions \eqref{e:comp} thus form a (nonhomogeneous) random walk on the interval.

The dynamics of the step skew product system depends on the Lyapunov exponents at the boundaries
$\Sigma_2^+ \times \{0\}$ and $\Sigma_2^+ \times \{1\}$.
We list the possibilities, which will be worked out in subsequent sections below. 

\setlist[description]{font=\normalfont\itshape} 

\begin{description}

\item[Intermingled basins]  With negative Lyapunov exponents at the boundaries, these boundaries are attracting. Their basins are intermingled: any open set in $\Sigma_2^+ \times \mathbb{I}$ intersects both basins. 

\item[Master-slave synchronization] With positive Lyapunov exponents at the boundaries, the boundaries are repelling.
 We find that for almost all fibers $\{\omega\} \times \mathbb{I}$, orbits of points in the same fiber converge to each other, i.e. synchronize.

\item[On-off intermittency] A zero Lyapunov exponent at a boundary makes that boundary neutral. With the other boundary repelling, 
a typical time series has long laminary phases where the orbit is close to the "off" state
(the neutral boundary) and has bursts where the orbit is in the "on" state, i.e. away from the neutral boundary. 
Orbits spend a portion of its iterates with full density near the neutral boundary. 

With two neutral boundaries, orbits spend a portion of its iterates with full density near the union of the neutral boundaries.

\item[Random walk with drift] With one attracting and one repelling or neutral boundary, most orbits approach the attracting boundary. 
\end{description}
Thus we find  the most elementary case of the more widespread phenomenon of  
intermingled basins \cite{aleyoryoukan92,kan94}, or on-off intermittency \citep{plsptr93,heaplaham94}, or
master-slave synchronization \citep{peccar90,sta97}.

%

The setup chosen in this paper is a starting point for research in random dynamics, see e.g. \cite{arn98}, and nonhyperbolic dynamics, see e.g. \cite{bondiavia05}, and has relations to nonautonomous dynamics, see e.g. \cite{kloras11}.
The following directions for generalizations give an idea of the many possibilities.
We will not give details, but refer to  \citep{arn98,bondiavia05,kloras11} for more.     
One may consider other measures than Bernoulli measures
to pick random compositions of the interval maps.
A natural generalization is also to let the diffeomorphisms on $\mathbb{I}$ depend on $\omega$ more generally than through
$\omega_0$ alone;
\[
(\omega,x) \mapsto (\sigma \omega , f_{\omega} (x) ).
\] 
One can further consider parameters $\omega$ from other spaces than symbol spaces, with other dynamics than generated by the shift 
operator. One may then also generalize the skew product structure to maps on fiber bundles, and study 
perturbations that destroy the skew product structure.
A heuristic principle going back to \cite{gorily99} states that phenomena in random dynamics on compact manifolds
may also occur for diffeomorphisms of manifolds of higher dimensions.

This paper is organized as follows.
We start with a section that contains definitions.
The next sections form the heart of the paper, describing possible dynamics for 
the considered class of step skew product systems.
An important role in the study of skew product systems is by invariant measures. A basic result gives the connection 
between invariant measures for skew product systems and their natural extensions.
In the appendix this is worked out in the simple context of step skew product systems over one-sided and two-sided shifts.\\

\noindent {\bf Acknowledgment.} 
We are indebted to Todd Young for discussions on the topics of this paper, and to Abbas Ali Rashid and Vahatra Rabodonandrianandraina 
for a careful reading and remarks.
Frank den Hollander pointed out relevant work by Lamperti on random walks.
Detailed comments from referees have been very helpful in improving the presentation.

\section{Step skew product systems}\label{s:sps}

This section serves to present the setup of this paper and to collect necessary definitions.
A skew product system is a dynamical system generated by a map $F: Y \times X \to Y \times X$ of the form
\begin{align}\label{e:sps}
 F (y,x) &= ( g(y) ,  f(y,x));
\end{align}
if one sees $X$ as the state space of interest, one has dynamics of the $x$ variable that is governed by the map $f$ 
which depends on the variable $y$ that changes through $g$.
The space $Y$ is the base space, the sets $\{y\} \times X$ are fibers.

We have an interest in skew product systems over full shifts.
Write $\Omega$ for the finite set of symbols $\{ 1,\ldots,N\}$.
Let $\Sigma_N = \Omega^\mathbb{Z}$ be the set of bilateral sequences $\omega=(\omega_n)^{\infty}_{-\infty}$ composed of symbols in $\Omega$.
Let $\sigma:\Sigma_N \to \Sigma_N$ be the shift operator;
the map $\sigma$ shifts every sequence $\omega \in \Sigma_N$ one step to 
the left, $(\sigma \omega)_i = \omega_{i+1}$. 
We can also consider the shift operator $\sigma$ acting on the one-sided symbol space $\Sigma_N^+$, i.e.
the 
space of sequences $\omega=(\omega_n)^{\infty}_0$ composed of symbols in $\Omega$.
The 
spaces $\Sigma_N$ and $\Sigma_N^+$ are endowed with the product topology.
This topology is generated by cylinders like
 $C^{k_1,\ldots,k_n}_{\omega_1,\ldots,\omega_n}$ for $\Sigma_N$,
\begin{align*}
 C^{k_1,\ldots,k_n}_{\omega_1,\ldots,\omega_n} = \{\omega' \in \Sigma_N \; ; \; \omega'_{k_i}= \omega_{k_i},\; \forall i=1,\ldots,n \}.
\end{align*}
As it will not lead to confusion, 
we use the same notation for cylinders in $\Sigma_N^+$.

Now let $M$ be a compact manifold, or compact manifold with boundary, and for $\omega \in \Sigma_N$, let $f_\omega: M \to M$
be diffeomorphisms depending continuously on $\omega$.
Consider skew product systems
$F: \Sigma_N \times M \to \Sigma_N \times M$; 
\[
F (\omega , x) = (\sigma \omega, f_\omega (x)).
\]

\begin{definition}
A skew product system $F:\Sigma_N \times M \to \Sigma_N \times M$ is  a step skew product system if it is of the form
\[
F (\omega , x) = (\sigma \omega, f_{\omega_0} (x)),
\]
i.e. the fiber maps depend on $\omega_0$ alone.
\end{definition}

We denote iterates of a skew product system
$F(\omega,x) = (\sigma \omega,  f_\omega (x))$ as
\[
F^n (\omega,x) = (\sigma^n \omega , f^n_\omega (x)).
\] 
Here, for $n \ge 1$, 
\[
f^n_\omega(x)= f_{\sigma^{n-1} \omega} \circ \cdots \circ f_{\omega} (x).
\]
For a step skew product system this becomes
\[
f^n_\omega(x)= f_{\omega_{n-1}} \circ \cdots \circ f_{\omega_0} (x).
\]
Observe that, if $-n <0$,
\[
f^{-n}_\omega (x) = \left(  f^n_{\sigma^{-n} \omega}  \right)^{-1}.
\] 

We also consider (step) skew products over the shift on one-sided symbol sequences.
We write 
$F^+$ for the skew product system 
\[
 F^+ (\omega,x) = (\sigma \omega , f_\omega (x))
\]
on $\Sigma_N^+ \times M$.
Recall that a natural extension of a continuous map is the smallest invertible extension, up to topological semi-conjugacy.
The skew product system $F$ on $\Sigma_N \times M$ is the natural extension of $F^+$ on $\Sigma_N^+ \times M$.  

\begin{definition}
Let $\mathbb{F}$ be a family of diffeomorphisms on $M$.
The iterated function system  $\text{IFS}\,(\mathbb{F})$ is the action of the semigroup generated by $\mathbb{F}$. 
\end{definition}

So a collection of  diffeomorphisms $f_i$, $1 \le i \le N$, generates an iterated function system.
And an iterated function system $\text{IFS}\,\{ f_1,\ldots,f_N\}$ on $M$ corresponds to a step skew product system
$F^+(\omega, x) = (\sigma \omega, f_{\omega_0} (x))$ on $\Sigma^+_N \times M$. 
Given an iterated function system $\text{IFS}\,(\mathbb{F})$,
a sequence $\{x_n : n\in \mathbb{N} \}$ is called a branch of an orbit of  $\text{IFS}\,(\mathbb{F})$ if 
for each $n\in \mathbb{N}$ there is $f_n\in \mathbb{F}$ such that $x_{n+1}=f_n(x_n)$. 
We say that $\text{IFS}\,(\mathbb{F})$ is minimal if every orbit has a branch which is dense in $M$. 

The appendix collects definitions and basic results on stationary and invariant measures in the context of
step skew product systems over shifts. We will make use of the material from the appendix in the following sections. 

\subsection{Interval fibers}

Focus of this paper is the following class of 
step skew products of diffeomorphisms on $\mathbb{I} = [0,1]$ over the full shift on two symbols $\{1,2\}$, earlier presented in the introduction.

\begin{definition}
Let $\mathcal{S}$ be the set of step skew product systems
${F^+}: \Sigma_2^+ \times \mathbb{I} \to \Sigma_2^+ \times \mathbb{I}$ with
\[
{F^+} (\omega,x) = (\sigma \omega , f_{\omega_0} (x)),
\]
where  $f_1, f_2$ are $C^2$ diffeomorphisms that fulfill the following conditions:
\begin{enumerate}
\item $f_i (0) = 0$, $f_i(1) = 1$ for $i=1,2$;
\item $f_1(x) < x$ for $x \in (0,1)$;
\item $f_2 (x) > x$ for $x \in (0,1)$.
\end{enumerate}
\end{definition}

So $f_1$ moves points in $(0,1)$ to the left, whereas $f_2$ moves points in $(0,1)$ to the right.
On $\Sigma_2^+$ we take Bernoulli measure $\nu^+$ where  the symbols $1,2$ have probability $p_1, p_2$, see Appendix~\ref{s:invmeasures}.

The (fiber) Lyapunov exponent of ${F^+}$ at a point $(\omega, x) \in \Sigma_2^+ \times \mathbb{I}$ is
\begin{equation*} 
\lim_{n\to \infty} \frac{1}{n} \ln \left( f_{\omega_{n-1}}'\left(f^{n-1}_\omega(x)\right)\cdots f_{\omega_0}'(x) \right) = 
\lim_{n\to \infty} \frac{1}{n} \sum_{i=0}^{n-1} \ln \left( f_{\sigma ^i\omega}' (f^i_\omega (x)) \right),
\end{equation*}
in case the limit exists.
Since  $x=0,1$ are fixed points of $f_i$, $i=1,2$, 
by Birkhoff's ergodic theorem, we obtain for $x=0,1$ that  
\[ 
L(x) = \lim_{n\to \infty} \frac{1}{n} \sum_{i=0}^{n-1} \ln \left( f_{\sigma ^i\omega}' (x) \right)= \int_{\Sigma^+_2} \ln \left( f_\omega'(x) \right) \, d\nu^+ (\omega)= \sum_{i=1}^2 p_i \ln \left(f_i'(x)\right)
\]
for $\nu^+$-almost all $\omega \in \Sigma_2^+$.

\begin{definition}
The standard measure $s$ on $\Sigma_2^+ \times \mathbb{I}$ is the product of Bernoulli measure $\nu^+$ and Lebesgue measure on $\mathbb{I}$.
\end{definition}

\begin{figure}[phtb]

\begin{center}
 \includegraphics[height=5.2cm]{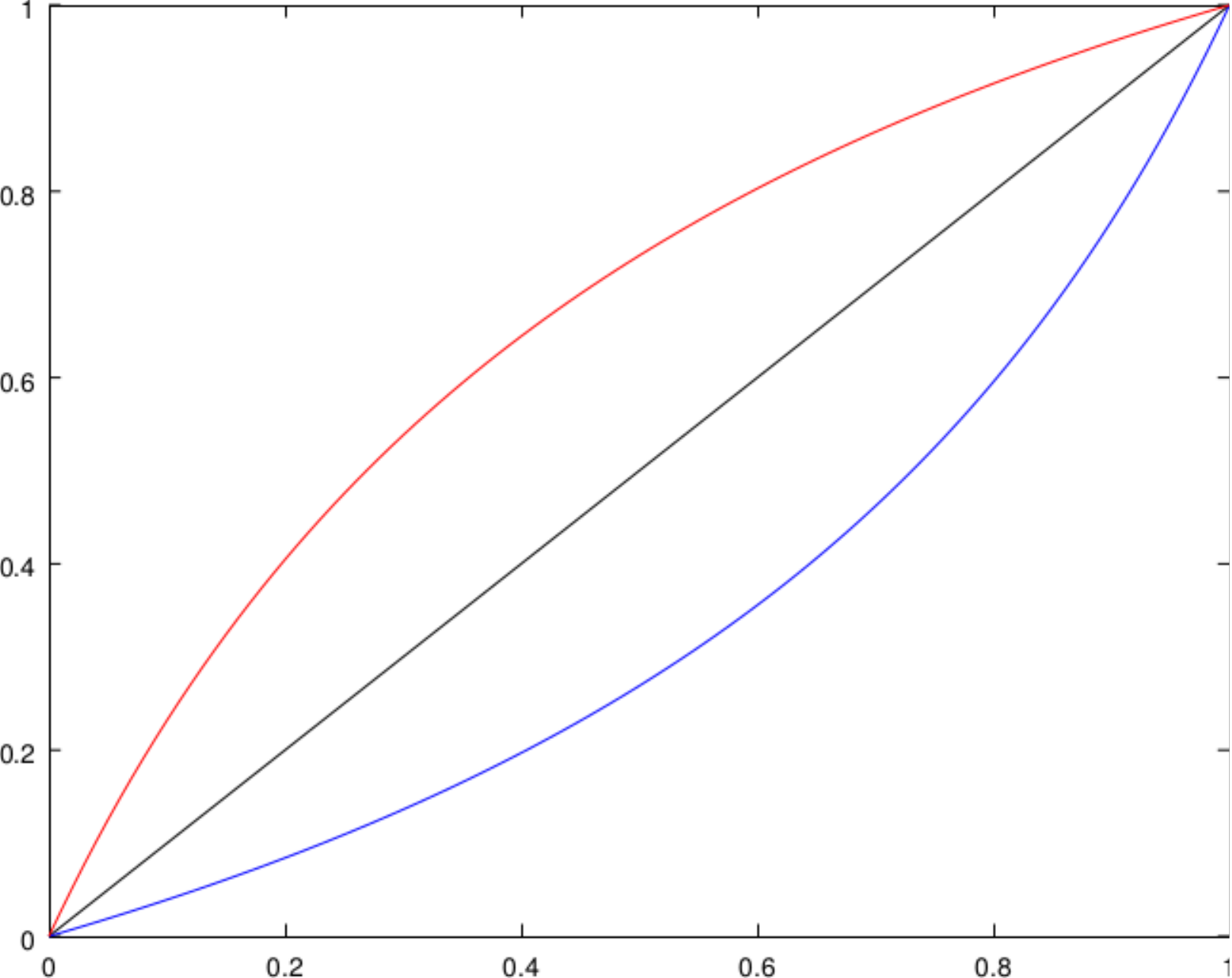}\hspace{0.4cm} 
 \includegraphics[height=5.2cm]{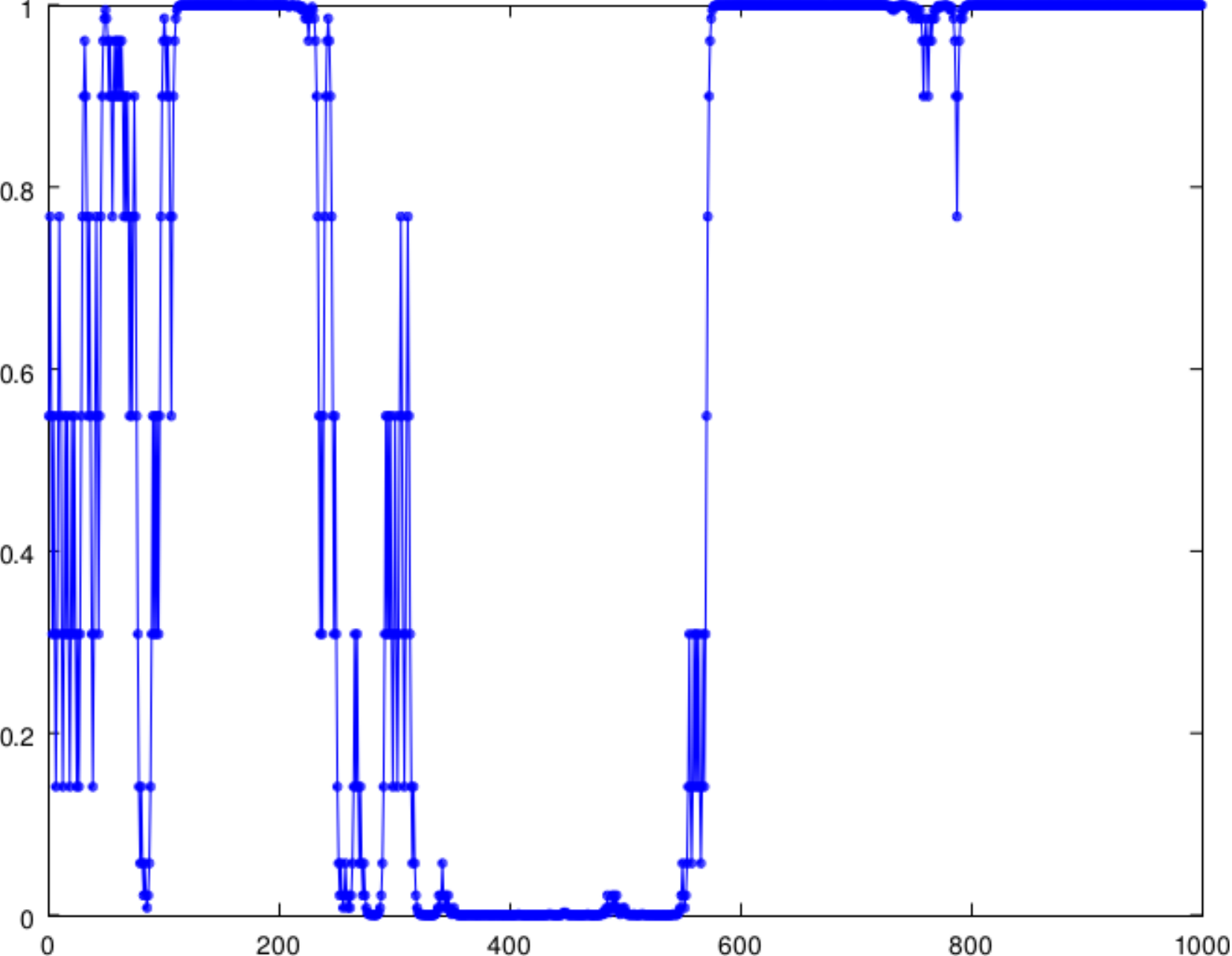} 
\caption{The left frame depicts the graphs of $g_1,g_2$, the diffeomorphisms on $\mathbb{I}$ that are conjugate to the maps $y \mapsto y \pm 1$ that generate the symmetric random walk. The right frame shows a time series of the iterated function system generated by $g_1,g_2$, both picked with
probability $1/2$.
\label{f:walk}}
\end{center}

\end{figure}
A specific example of a step skew product system from $\mathcal{S}$ comes  from the symmetric random walk.
The symmetric random walk is given by translations 
\begin{align*}
k_1(x) &= x-1,\\
k_2(x) &= x+1
\end{align*}
on the real line, where both maps are chosen randomly with probability $1/2$.
Now conjugate the symmetric random walk to maps on $\mathbb{I}$ as follows.  
Consider the coordinate change given by the diffeomorphism
$h: \mathbb{R} \to (0,1)$, 
\[
h(x) = \frac{e^x}{1 + e^x}
\]
(note that $h^{-1} (x) = \ln (x/(1-x)))$.
Define the step skew product system on $\Sigma_2 \times \mathbb{I}$ generated by the fiber diffeomorphisms $g_i =h \circ k_i \circ h^{-1}$ with 
$g_i(0)=0$, $g_i(1)=1$, $i=1,2$.
We have
\begin{align}
\label{e:g1}
g_1 (x) &=  \frac{\frac{1}{e} x}{1 + (\frac{1}{e} -1)x},
\\
\label{e:g2}
g_2 (x) &=  \frac{e x}{1 + (e-1)x},
\end{align}
see Figure~\ref{f:walk}.
We will also refer to the step skew product system generated by $g_1,g_2$ as the symmetric random walk.

Observe $g_1'(0) = \frac{1}{e}$, $g_2'(0) = e$, $g_1'(1) = e$, $g_2'(1) = \frac{1}{e}$,
so that the symmetric random walk has zero Lyapunov exponents at the boundaries $\Sigma_2^+ \times \{0\}$ and
$\Sigma_2^+ \times \{1\}$.
Perturbations of $g_1, g_2$ that preserve the boundary points $0,1$ lead to diffeomorphisms $f_1,f_2$  with various signs of Lyapunov exponents at the boundaries: all cases that are treated in the following sections also occur as small perturbations from the symmetric random walk.

The text book \cite{chu68} contains a discussion of recurrence properties of random walks on the line with i.i.d. steps.
In the same vein one can ask for the iterated function system $\text{IFS}\,(\{f_1,f_2\})$ to be minimal on $(0,1)$.   
The proof of \cite[Lemma~3]{ily10} gives the following result.

\begin{proposition}\label{p:ily10}
Assume that $\lambda = f_1'(0) < 1$, $\mu = f_2'(0) > 1$. 
Assume further that either
\[
\ln (\lambda) / \ln (\mu) \not \in \mathbb{Q},
\]
or
\[
\frac{f_1'' (0)}{\lambda^2 - \lambda} \neq \frac{f_2''(0)}{\mu^2 - \mu}.
\]
Then the iterated function system generated by $f_1,f_2$ is minimal on $(0,1)$.  
Such minimality is also implied by analogous conditions at the end point $1$.
\end{proposition}

\begin{proof}
For the proof we refer to \cite{ily10}. We add some comments to clarify the conditions.
Il'yashenko \cite[Lemma~3]{ily10} considers, for $x,y \in (0,1)$, compositions $f_2^l \circ f_1^k (x)$ that converge to $y$ for suitable $k,l \to \infty$. Note that this property implies minimality. 
His analysis uses linearizing coordinates $h \circ f_1 \circ h^{-1} (x) = \lambda x$ with $x \in [0,s]$ for an $s <1$. Here $h$ is a local diffeomorphism. 
The two cases where $\ln (\lambda), \ln(\mu)$ are rationally dependent or not, are distinguished.
In case $\ln (\lambda), \ln(\mu)$ are rationally dependent, 
the argument works if the second order derivative of $h \circ f_2 \circ h^{-1}$ at  $0$ is not zero.
An explicit calculation shows that this gives the condition in the proposition.
\end{proof}

Obviously, the iterated function system 
generated by $g_1$ and $g_2$, where $g_2 = g_1^{-1}$, is not minimal. 

\section{Intermingled basins}\label{s:intermingled}

Kan \cite{kan94} describes an example of a skew product system on $\mathbb{T} \times \mathbb{I}$, over an expanding circle map in the base, 
where the boundary components 
$\mathbb{T} \times \{0\}$ and $\mathbb{T}\times \{1\}$ are attractors so that both basins intersect each open set.
We will describe his results in the elementary setting of step skew product systems.


The following result describes intermingled basins for step skew product systems ${F^+} \in \mathcal{S}$.

\begin{theorem}\label{t:kan}
Let ${F^+} \in \mathcal{S}$ and assume $L(0)<0$ and $L(1)<0$.
The sets $\Sigma_2^+ \times \{0\}$ and $\Sigma_2^+ \times \{1\}$  attract sets of positive standard measure.
 Both their  basins lie dense in $\Sigma_2^+ \times \mathbb{I}$. The union of the basins has full standard measure. 
 
 Let $F: \Sigma_2 \times \mathbb{I} \to \Sigma_2 \times \mathbb{I}$ denote the natural extension of ${F^+}$.
 There is an invariant measurable graph $\xi: \Sigma_2 \to \mathbb{I}$ that separates the basins: 
  for $\nu$-almost all $\omega$,
\[ \lim_{n\to\infty} f^n_\omega (x) = \left\{ \begin{array}{ll} 0, & \text{ if } x < \xi(\omega),
                                                      \\ 
                                                                1, & \text{ if } x > \xi(\omega).
                                              \end{array} \right.
\]
\end{theorem}

We note that $\Sigma_2^+ \times \{0\}$ and $\Sigma_2^+ \times \{1\}$ are attractors in Milnor's sense \cite{mil85}.
The values $\xi (\omega)$ depend only on the present and future coefficients $(\omega_i)_0^\infty$.
Before starting the actual proof, we provide a simple argument showing positive standard measure of the basins of
$\Sigma_2^+ \times \{0\}$ and $\Sigma_2^+ \times \{1\}$.  

\begin{lemma}\label{l:bonmil}
Let ${F^+} \in \mathcal{S}$ and assume $L(0)<0$.
Let 
\begin{align*}
r(\omega) = \sup \{  x\in \mathbb{I} \;  \mid  \; \lim_{n\to \infty} f^n_\omega (x) = 0 \}.
\end{align*}
Then $r (\omega)>0$ for $\nu^+$-almost all $\omega \in \Sigma_2^+$.  
\end{lemma}

\begin{proof}
The argument follows \cite[Lemma~2.2]{kan94} or \cite[Lemma~A.1]{bonmil08}.
For any $\varepsilon > 0$ there exists $\delta>0$ so that 
\[
f_i'(x) \le f_i'(0) + \varepsilon 
\]
if $x < \delta$, for both $i=1,2$.
Write $a_i = \ln (f_i'(0) + \varepsilon)$. 
Recall that $L(0) = p_1 \ln (f_1'(0)) + p_2 \ln (f_2'(0))$ is negative by assumption.
By Birkhoff's ergodic theorem applied to the function $\omega \mapsto \ln (f_\omega' (0) + \varepsilon)$,
for $\nu^+$-almost all $\omega$,
\[
\lim_{n\to\infty} \frac{1}{n} \sum_{i=0}^{n-1} a_{\omega_i} = p_1 a_1 + p_2 a_2,
\] 
which is negative if $\varepsilon$ is small enough.
So, for $\nu^+$-almost all $\omega$, $\sum_{i=0}^{n-1} a_{\omega_i}$ goes to $-\infty$ as $n\to \infty$ and
\[A( \omega) = \max \{ 0, \max_{n\ge 1}  \sum_{i=0}^{n-1} a_{\omega_i} \} \]  exists.
Take $x_0 < \delta e^{-A(\omega)} \le \delta$.
 Then $x_n = f^n_\omega (x_0)$ satisfies
 \[
 x_n < e^{\sum_{i=0}^{n-1} a_{\omega_i} } e^{-A (\omega)} \delta \le \delta
 \] 
for all $n \ge 0$ and in fact $\lim_{n\to \infty} x_n = 0$. 
This proves the lemma.
\end{proof}

Since the function $r$ is positive almost everywhere, the basin of $\Sigma_2^+ \times \{0\}$ has positive 
standard measure. The same holds for the basin of $\Sigma_2^+ \times \{1\}$. 
It is easily seen that any open set in $\Sigma_2^+ \times \mathbb{I}$ intersects both basins; forward iterations must accumulate 
onto both $\Sigma_2^+ \times \{0\}$ and $\Sigma_2^+ \times \{1\}$ using that the shift operator is an expansion
and $0$ and $1$ occur as attracting fixed points for $f_1$ and $f_2$ respectively.

\begin{proof}[Proof of Theorem~\ref{t:kan}]
We prove the theorem by considering the inverse diffeomorphisms, i.e. a step skew product 
with positive Lyapunov exponents along $\Sigma_2^+ \times \{0\}$ and $\Sigma_2^+ \times \{1\}$.
For the duration of this proof, we consider $F^+ \in \mathcal{S}$ with $L(0) >0 $ and $L(1) >0$.
The following lemmas deal with this.
The theorem will follow by linking the derived statements on the natural extension $F$ of $F^+$ and the statements we wish to prove for its inverse. 
\begin{figure}[htbp]

\begin{center}
 \includegraphics[height=5.2cm]{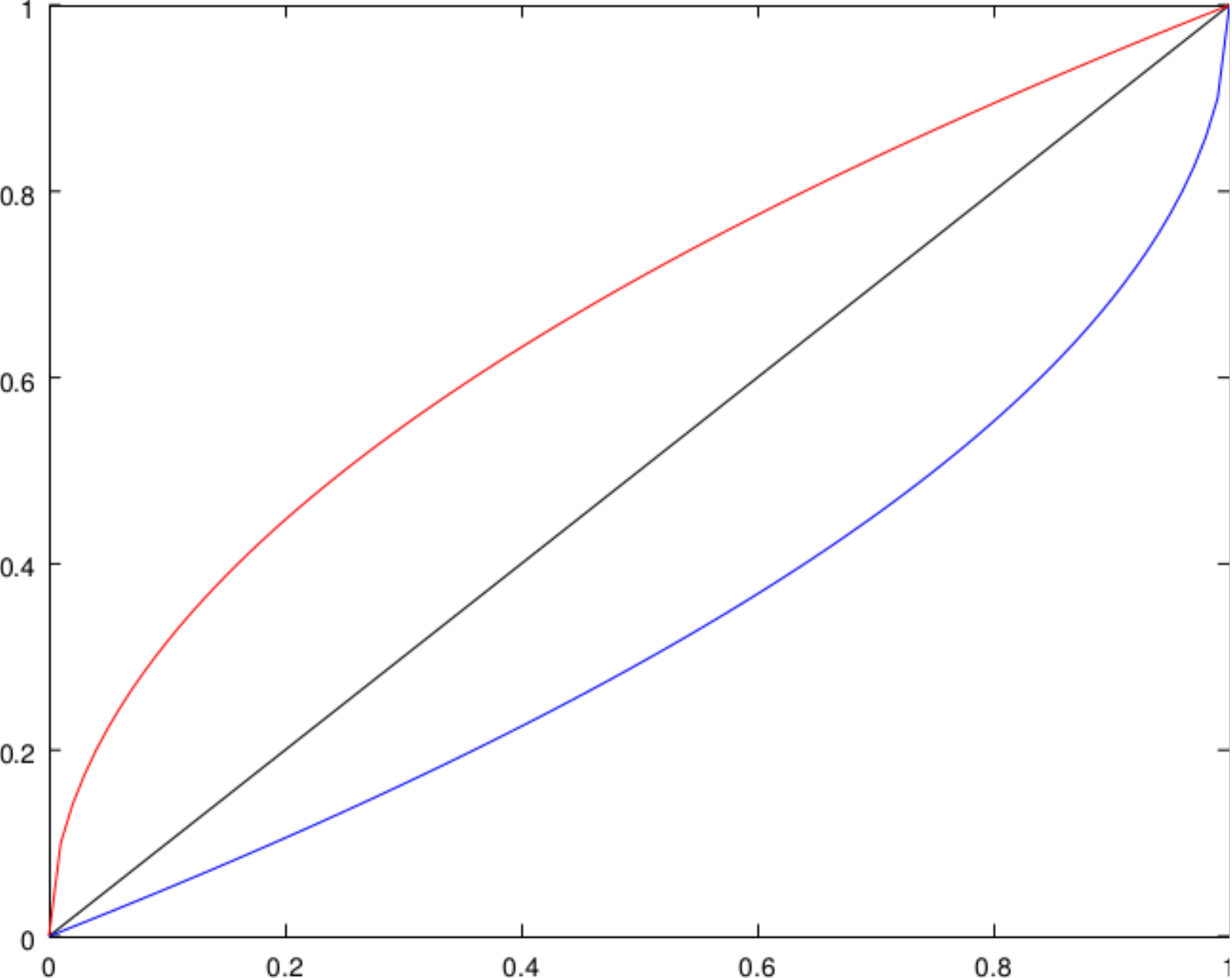}\hspace{0.4cm} 
 \includegraphics[height=5.2cm]{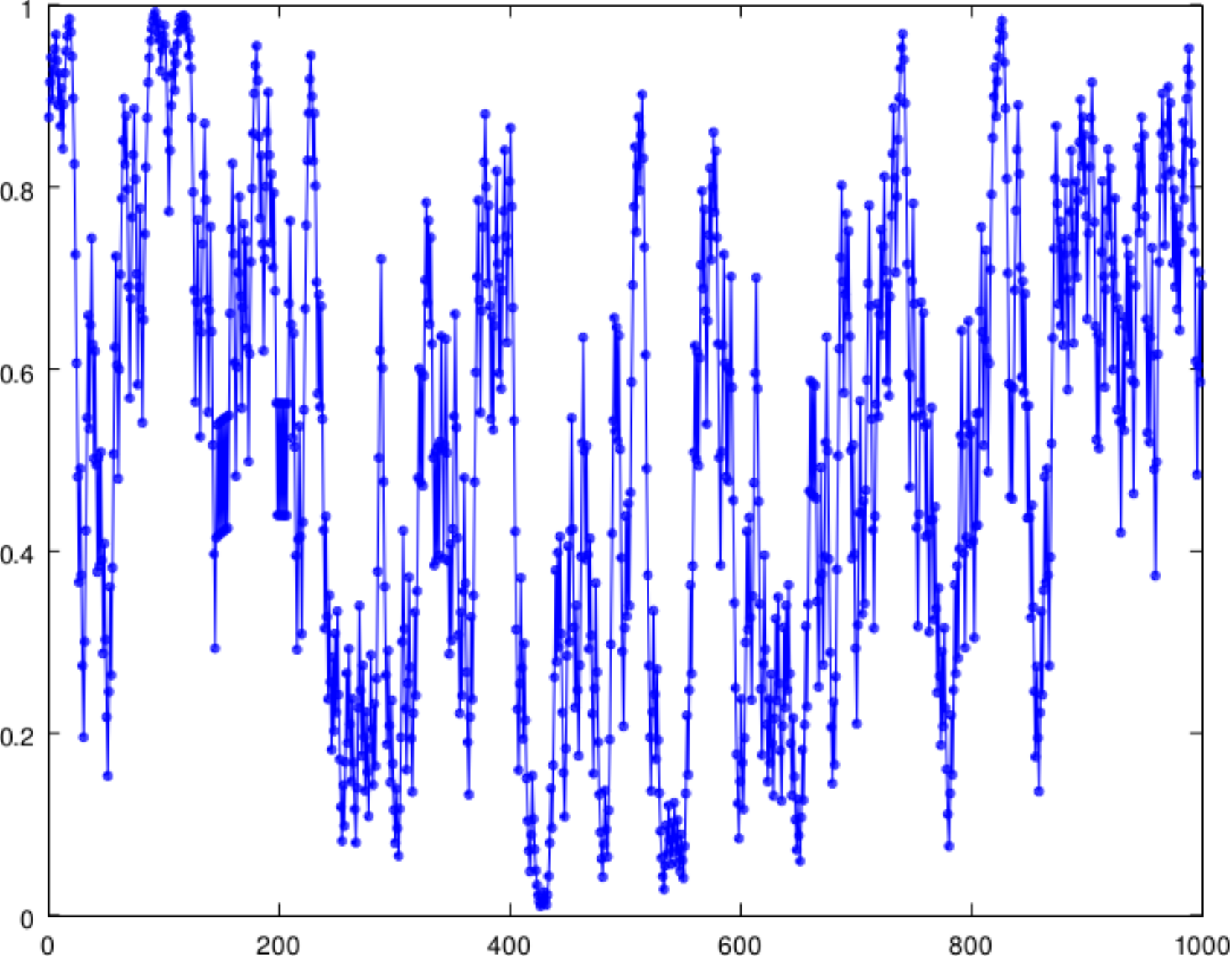} 
\caption{With $r = 1/2$, the diffeomorphisms  $f_1 (x) = x - r x (1-x)$ and $f_2 (x) = x + r x (1-x)$  (picked with probabilities $1/2$) 
give negative Lyapunov exponents at the end points $0,1$. Depicted, in the left frame, are the graphs of the inverse diffeomorphisms
 $f_1^{-1} (x) = \frac{1 - r - \sqrt{(1-r)^2 +4 r x}}{-2r}$ and 
$f_2^{-1} (x) = \frac{1 + r - \sqrt{(1+r)^2 -4 r x}}{2r}$.
The inverse maps give positive Lyapunov exponents at the end points. The right frame shows a time series for the iterated function system generated by  $f_1^{-1}$ and $f_2^{-1}$. 
\label{f:inversekan}}
\end{center}

\end{figure}

We write $\mathcal{P}_\mathbb{I}$ for the space of probability measures on $\mathbb{I}$ equipped with the weak star topology. 
As explained in Appendix~\ref{s:invmeasures}, a stationary measure is a fixed point of $\mathcal{T}: \mathcal{P}_\mathbb{I} \to \mathcal{P}_\mathbb{I}$ given by
\[
 \mathcal{T} m = p_1 f_1 m + p_2 f_2 m,
\]
where $f_i m$ is the push forward measure $f_i m (A) = m (f_i^{-1} (A))$. 
 
 \begin{lemma}\label{l:stat}
Let ${F^+} \in \mathcal{S}$ and  assume $L(0) >0$ and $L(1) >0$.
 Then there exists an ergodic stationary measure $m$ with $m (\{0\} \cup \{1\}) = 0$.
\end{lemma}

\begin{proof}
For 
small $ 0<\alpha<1$, $q>0$, and positive $c$, define
\begin{align*} 
\mathcal{N}_{c} &= \{m \in \mathcal{P}_\mathbb{I} \; ; \; 
\forall~ 0 \le  x \le  q,~ m\big([0,x)\big) \le c x^\alpha
\text{ and }~ m\big((1-x,1]\big) \le c x^\alpha
\}.
\end{align*}
The conditions exclude stationary measures 
supported on the end points $0$ or $1$.
Note that $\mathcal{N}_{c}$ depends on $\alpha$ and $q$;
but 
we do not include this dependence in the notation.
We first show that there exist $c>0$ and $\alpha,q>0$  close to $0$ 
 such that $\mathcal{T} (\mathcal{N}_{c}) \subset \mathcal{N}_c$.

Write $\rho_i = f_i'(0)$.
We claim that there is a small $ \alpha>0 $ such that the assumption $L(0)>0$ 
implies
$\sum_{i=1}^{2} p_i \rho_i^{-\alpha} < 1$.
Namely, since
$\mathop {\lim }\limits_{\alpha \to 0}{ \frac{1-\rho_i^{-\alpha}}{\alpha}}= \ln{\rho_i}$ for $ 1\le  i \le  2$,
$\sum_{i=1}^{2} p_i \ln{\rho_i} >0$ 
implies 
that for sufficiently small $\alpha>0$,
\[
\sum\limits_{i=1}^{2} {p_i}\frac{1-\rho_i^{-\alpha}}{\alpha}>0.
\]
Multiplying by $\alpha$ we get
\[
\sum\limits_{i=1}^{2} {p_i} - \sum\limits_{i=1}^{2} {p_i} \rho_i^{-\alpha}>0,
\]
which implies $\sum_{i=1}^2 {p_i} \rho_i^{-\alpha} <1$, because $\sum_{i=1}^{2}{p_i}=1$. 

Thus, there exists a small $\delta>0$ 
so that 
\begin{align}\label{e:smallerthanone}
\sum_{i=1}^{2} \frac{p_i}{(\rho_i-\delta)^{\alpha}} &<1.
\end{align}
Moreover,
for such $\delta>0$ we are able to choose a sufficiently small $q=q(\delta)>0$ 
so that
\begin{equation}\label{e:fineq}
f^{-1}_{i}(x) \le  \frac{x}{\rho_{i}-\delta},
\; \forall \; 0 \le  x \le  q.
\end{equation}

Take $c$ with \[c q^\alpha >1.\]
Note that this implies that in the definition of $\mathcal{N}_c$, 
 $m\big([0,x)\big) \le c x^\alpha$ and $m\big((1-x,1]\big) \le c x^\alpha$ for
 any $0\le x\le 1$, and not just for $0 \le x \le q$. 
Take a measure $m \in \mathcal{N}_c$.
To prove $\mathcal{T} m \in \mathcal{N}_c$,
we must show that for $ x \le  q$,  
$\mathcal{T} m  \big([0,x)\big) \le  c x^\alpha$. 
Knowing that 
$m \big([0,x)\big) \le  c x^\alpha$ and 
applying \eqref{e:smallerthanone}, \eqref{e:fineq} we  obtain the following estimates:
\begin{align}
\mathcal{T}  m  \big([0,x)\big) 
&=
\sum_{i=1}^2  p_{i} f_{i} m \big([0,x)\big)
=
\sum_{i=1}^2  p_{i} m\left(f^{-1}_{i}[0,x)\right)
\leq
\sum_{i=1}^2  p_{i} m  \left( [0 , \frac{x}{\rho_{i}-\delta}) \right) \nonumber 
\\
\label{e:Tinside}
&\leq
\sum_{i=1}^{2} p_i c \left(\frac{x}{\rho_i-\delta}\right)^\alpha 
=
c \left(\sum_{i=1}^{2} \frac{p_i}{(\rho_i-\delta)^\alpha}\right) x^\alpha 
\leq
cx^\alpha. 
\end{align}
Estimates near the right boundary point are treated in the same manner.

By the
Krylov-Bogolyubov averaging method, 
for a measure $m \in \mathcal{N}_c $ 
there
is a subsequence 
of 
$\{\frac{1}{n}\sum_{r=0}^{n-1} \mathcal{T}^{r} m\}_{n \in \mathbb{N}}$ that is convergent to a probability 
measure
$\hat{m} \in \mathcal{N}_c$
such that
$\mathcal{T} \hat{ m}=\hat{ m}$.

The following additional reasoning shows that there is an ergodic stationary measure in $\mathcal{N}_c$.
The set of stationary measures $\mathcal{M}_\mathbb{I}$ is a convex compact subset of $\mathcal{P}_\mathbb{I}$.
The ergodic stationary measures are the extreme points of it.
Note that $\mathcal{N}_c \cap \mathcal{M}_\mathbb{I}$ is a convex compact subset of  $\mathcal{M}_\mathbb{I}$, which is itself 
also convex and compact.
We claim that the extreme points of $\mathcal{N}_c\cap \mathcal{M}_\mathbb{I}$ are also extreme points of $\mathcal{M}_\mathbb{I}$.
Suppose by contradiction that there are $n_1,n_2 \in \mathcal{M}_\mathbb{I} \setminus (\mathcal{N}_c\cap \mathcal{M}_\mathbb{I})$ 
and the convex combination $m = s n_1 + (1-s) n_2 \in \mathcal{N}_c\cap \mathcal{M}_\mathbb{I}$.
In this case, for $0\le x \le q$,
$n_1 ([0,x)) \le (c/s) x^\alpha$ and $n_1 ( (1-x,1] )  \le (c/s) x^\alpha$ and similar estimates for $n_2$.
 That is,  $x \mapsto n_i ( [0,x) ) / x^\alpha $ and
$x \mapsto n_i ( (1-x,1] ) / x^\alpha $ are bounded.
As $\mathcal{T} m = m$, we have by \eqref{e:smallerthanone}, \eqref{e:Tinside} that $m \in \mathcal{N}_{\tilde{c}}$ for some $\tilde{c} < c$.
It follows that $t n_1 + (1-t) n_2 \in \mathcal{N}_c\cap \mathcal{M}_\mathbb{I}$ for $t$ close to $s$. 
So $s$ is an interior point of the set of values $t$ for which $t n_1 + (1-t) n_2 \in \mathcal{N}_c\cap \mathcal{M}_\mathbb{I}$. 
Since $\mathcal{N}_c\cap \mathcal{M}_\mathbb{I}$ is closed it follows that $n_i \in \mathcal{N}_c\cap \mathcal{M}_\mathbb{I}$ and the claim is proved.  
Since the extreme points of $\mathcal{M}_\mathbb{I}$ are ergodic stationary measures, we conclude that 
the extreme points of $\mathcal{N}_c\cap \mathcal{M}_\mathbb{I}$ are ergodic stationary measures.  
Since the Krein-Milman theorem the set of extreme points of $\mathcal{N}_c \cap \mathcal{M}_\mathbb{I}$ 
is nonempty, there are ergodic stationary measures in $\mathcal{N}_c$.
\end{proof}

A stationary measure $m$ gives an invariant measure $\mu_m$ for the step skew product system $F$, as explained in 
Appendix~\ref{s:invmeasures}. Its conditional measures on fibers $\{\omega\} \times \mathbb{I}$ are denoted by $\mu_{m,\omega}$.
 
\begin{lemma}\label{l:dmeas}
For every ergodic stationary probability measure $m$, 
the conditional measure $\mu_{m,\omega}$ of $\mu_m$
is
a $\delta$-measure for $\nu$-almost every $\omega \in \Sigma_2$.
\end{lemma}

\begin{proof}
We follow \cite[Theorem 1.8.4]{arn98}.
Consider a $\mu_m$ and its conditional measures $\mu_{m,\omega}$.
Let 
$X_m(\omega)$ be the smallest median of $\mu_{m,\omega}$,
i.e. 
the infimum of all points $x$ for which 
 \[
\mu_{m,\omega}([0,x]) \geq \frac{1}{2}~\text{and}~\mu_{m,\omega}([x,1])\geq \frac{1}{2}.
 \]
The set of medians of $\mu_m$ is a compact interval 
and 
$X_m: \Sigma_{2} \rightarrow \mathbb{I}$ is measurable.
Define
$J^-_m(\omega) = [0,X_m(\omega)]$ for which
by definition 
$\mu_{m,\omega}( J^-_m(\omega))\geq \frac{1}{2}$.
The 
set $J^-_m(\omega)$ is invariant:
since $f_1$ and $f_2$  are increasing, for every $x_1 < x_2$ and $\omega$
we have $f_\omega (x_1)< f_\omega (x_2)$.
This implies that $x$ is a median of $\mu_{m,\omega}$ if and only if $f_\omega (x)$
is a median of $f_\omega \mu_{m,\omega}$. By invariance of $\mu_m$ for $F$
we 
have \[f_\omega \mu_{m,\omega}= \mu_{m,\sigma\omega}.\]
Hence, 
$X_m(\sigma \omega) = f_\omega(X_m(\omega))$ 
which 
implies $J^-_m(\sigma \omega)=f_\omega( J^-_m(\omega))$.

Because $\mu_{m}$ is ergodic and $J^-_m(\omega)$ is 
invariant, 
$\mu_{m,\omega}(J^-_m(\omega))=1$, $\nu$-almost surely.
Applying 
the same argument to $J^+_m(\omega) = [X_m(\omega),1]$,
we obtain 
\[\mu_{m,\omega}(\{X_m(\omega)\}) =1\]
for 
$\{ X_m(\omega)\}= J^-_m(\omega) \cap J^+_m(\omega)$. 
Thus
$\mu_{m,\omega}= \delta_{X_m(\omega)}$ for $\nu$-almost every $\omega \in \Sigma_{2}$.
\end{proof}

The following lemma shows that the set of stationary measures is the triangle consisting 
of convex combinations of $\delta_0$, $\delta_1$ and one other ergodic stationary measure $m$. 

\begin{lemma}\label{l:unique}
There is a unique stationary measure $m$ with $m(\{0\} \cup \{1\}) = 0$.
\end{lemma}

\begin{proof}
Suppose there are two different such stationary measures $m_1$, $m_2$.
We may take $m_1$, $m_2$ to be ergodic stationary measures.
This corresponds to two ergodic invariant measures $\mu_{m_1} \neq \mu_{m_2}$ for $F$.
By Proposition~\ref{p:fapp} and Lemma~\ref{l:dmeas} there are measurable 
functions 
$X_{m_i}:\Sigma_{{2}}\rightarrow \mathbb{I}$
and sets
$D_i \subset \Sigma_{{2}}$ with $\nu(D_i)=1$, for $i=1,2$, such 
that
\[
\lim_{n\to \infty} f^n_{\sigma^{-n} \omega} m_i=\delta_{X_{m_i}(\omega)} 
\]
for every $\omega \in D_i$. From $\nu(D_i)=1$ we have 
$\nu (D_1 \cap D_2) =1$.
Since $\mu_{m_1},\mu_{m_2}$ are mutually singular we have that for 
 a generic
$\bar{\omega} \in D_1 \cap D_2$,  
$
X_{m_1}(\bar{\omega}) \neq X_{m_2}(\bar{\omega})
$.
%
Without loss of generality suppose that $X_{m_1}(\bar{\omega})< X_{m_2}(\bar{\omega})$.

Observe that the supports of $m_1$ and $m_2$ are invariant:
\[
\mathrm{supp}\, (m_i) = f_1 (\mathrm{supp}\, (m_i)) \cup f_2 (\mathrm{supp}\, (m_i))
\] 
for $i=1,2$. 
The convex hulls of the supports of $m_1$ and $m_2$ therefore both equal $\mathbb{I}$.
We can find generic points $(\bar{\omega},x_1)$ and $(\bar{\omega},x_2)$  
for $m_1$ and $m_2$ such that $x_1 > x_2$. Because  
\[ \lim_{n\to \infty} f^n_{\sigma^{-n} \bar{\omega}}(x_i)   = X_{m_i}(\bar{\omega})\]
and  $f_1$, $f_2$  are both strictly increasing, we conclude 
that 
$X_{m_2}(\bar{\omega})< X_{m_1}(\bar{\omega})$,  contradicting our assumption.
Thus, 
$\mu_{m_1} = \mu_{m_2}$.
\end{proof}

By Proposition~\ref{p:fapp}, Lemmas~\ref{l:dmeas} and \ref{l:unique}, there exists
a measurable function $\xi: \Sigma_2 \to \mathbb{I}$
such that
\[
\lim_{n\to \infty} f^n_{\sigma^{-n} \omega} m = \delta_{\xi(\omega)}, 
\]
for $\nu$-almost all $\omega$,
where $m$ is the stationary measure with $m(\{0\} \cup \{1\}) =0$.
As the convex hull of the support of $m$ equals $\mathbb{I}$ and $f_1$, $f_2$ are increasing,
this implies 
\[
\lim_{n\to \infty} f^n_{\sigma^{-n} \omega}(x) =\xi(\omega)
\]
for every $x \in (0,1)$.
Again since the diffeomorphisms $f_1$, $f_2$ are increasing, 
for the inverse diffeomorphisms 
$(f^n_{\sigma^{-n}\omega})^{-1} = f^{-1}_{\omega_{-n}} \circ \cdots \circ f^{-1}_{\omega_{-1}}$   this yields
\[
\lim_{n\to\infty} (f^n_{\sigma^{-n}\omega})^{-1} (y) = 1
\]
if $y > \xi(\omega)$ and
\[
\lim_{n\to\infty} (f^n_{\sigma^{-n}\omega})^{-1} (y) = 0
\]
if $y < \xi(\omega)$.
 Note that this also shows that for the inverse diffeomorphisms, the union of the basins of attraction
 of $\Sigma_2 \times \{0\}$ and $\Sigma_2 \times \{1\}$ has full standard measure.
\end{proof}

Let us give some pointers to further research literature: a discussion of a step skew product system over the full shift on two symbols, with piecewise linear fiber maps is in \cite{alsmis14}. 
Several articles discuss extensions to skew product systems that are not step skew product systems.
We refer the reader in particular to \citep{klesal11,ilyklesal08} and \cite[Section~11.1]{bondiavia05}.
Further studies that quantify the phenonomenon are  
\citep{ottsomalekanyor93,kel14}. 
See \citep{ily10,ily11,ghahom16} for related work on so-called thick attractors (attractors of positive standard measure).

\section{Master-slave synchronization}

The proof of Theorem~\ref{t:kan} relies on an analysis of step skew product systems with
positive Lyapunov exponents 
at the boundaries.
The following result further discusses such step skew product systems.
It describes a synchronization phenomenon that is illustrated in the right frame of Figure~\ref{f:inversekan2}.
 \begin{figure}[htbp]
 
 \begin{center}
  \includegraphics[height=5.2cm]{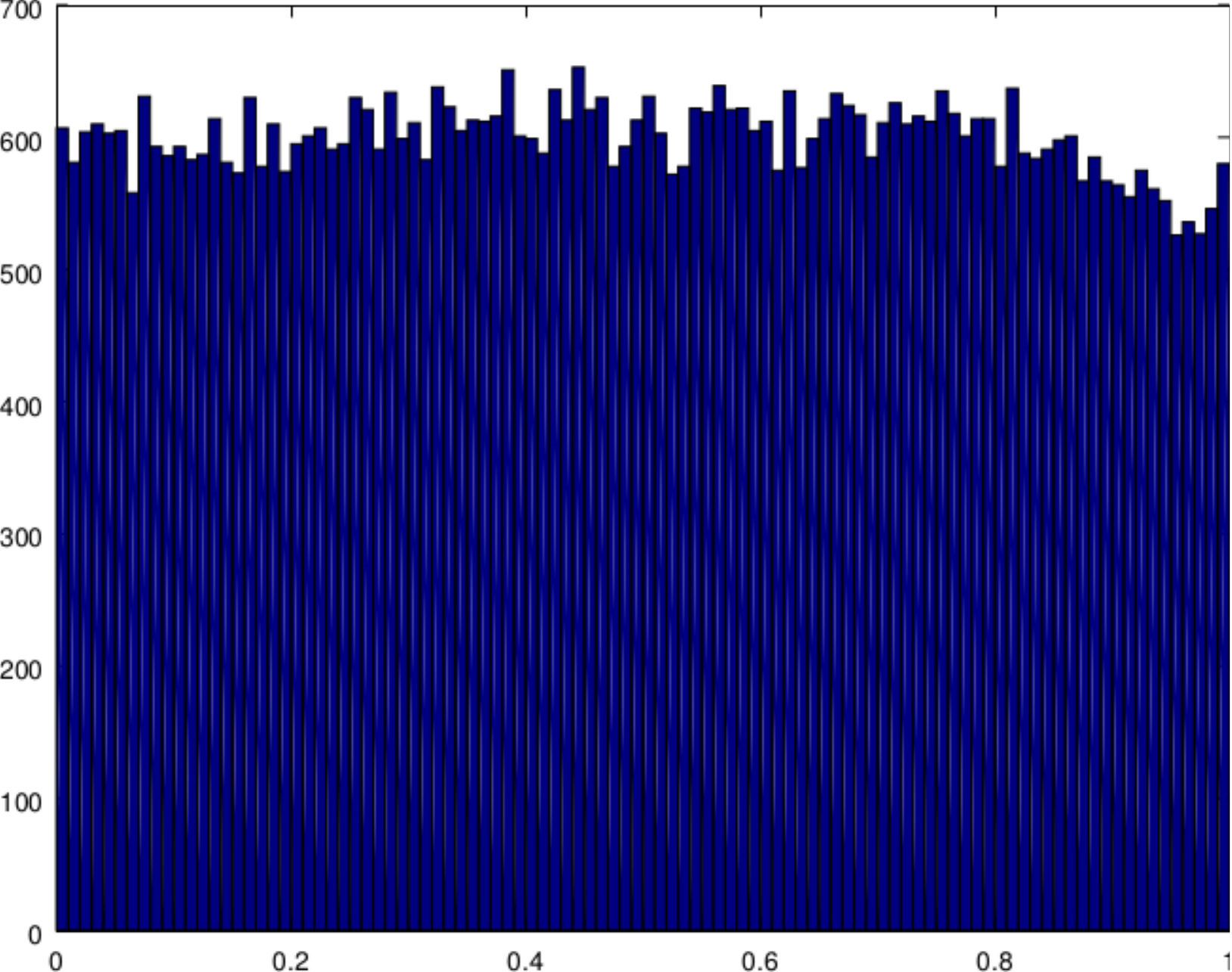}\hspace{0.4cm} 
  \includegraphics[height=5.2cm]{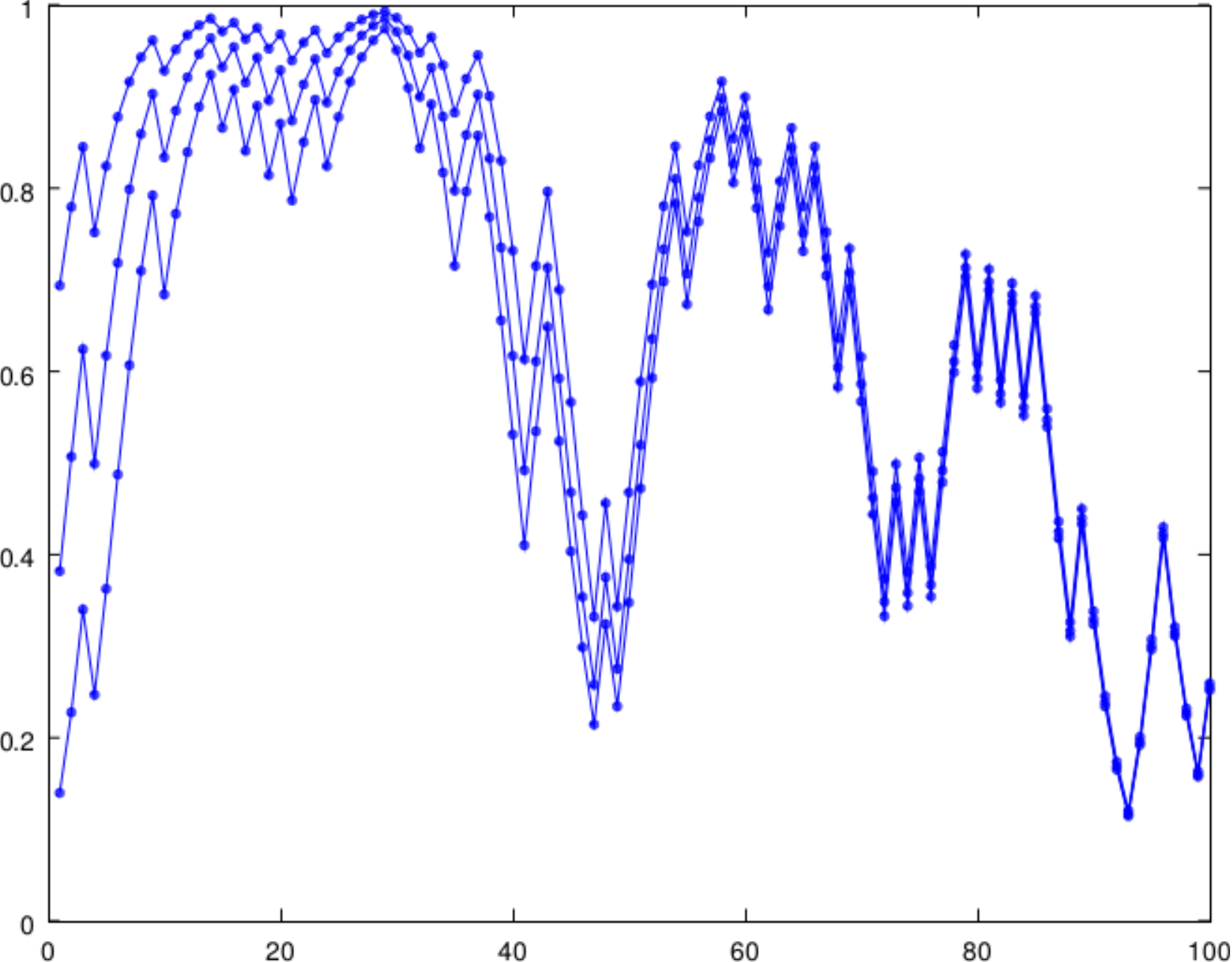} 
 \caption{The left frame shows a numerically computed histogram for a time series of the iterated function systems generated by
 the same diffeomorphisms $f_1^{-1}$ and $f_2^{-1}$ used in Figure~\ref{f:inversekan}.
 The right frame indicates asymptotic convergence of orbits within fibers: 
 it depicts time series for three different initial conditions in $\mathbb{I}$ with the same $\omega$.
 \label{f:inversekan2}}
 \end{center}
 
 \end{figure}
We see an example of master-slave synchronization, which refers to synchronization caused by external forcing. 
It is explained by a single attracting invariant graph
for the skew product system \citep{sta97}.
From a slightly different perspective one can also view this as synchronization by noise, where common noise synchronizes orbits with different initial conditions.

In a general context of skew product systems $F (y,x) = (g(y), f(y,x))$ on a product  $Y \times X$ of metric spaces $Y,X$, as in \eqref{e:sps},
master-slave synchronization is given by the following picture. If $\{ (y,\xi(y)) \;\mid\; y \in Y \}$ is a globally attracting graph, then the orbits $F^n (y,x_1)$ and $F^n (y,x_2)$
converge to each other. In particular if one observes the dynamics of the $x$-variable, one has
\begin{align*} 
\lim_{n\to \infty} d  ( f^n (y,x_1)   , f^n ( y,x_2))&=  0,
\end{align*} 
where $d$ is a metric on $X$ and $F^n (y,x) = (g^n(y) , f^n(y,x))$. 
An illustrative example, of linear differential equations forced by the Lorenz equations,  
is given by Pecora and Carroll in \cite{peccar90}.
We refer to \cite{pikroskur01} for an explanation of synchronization in a range of contexts.

The following result describes a similar effect for step skew product systems $F^+ \in \mathcal{S}$; the proof
employs a measurable invariant graph for the natural extension $F$ of $F^+$. 

\begin{theorem}\label{t:sync}
Let $F^+ \in \mathcal{S}$ and assume $L(0)>0$ and $L(1)>0$.
Let $x_0,y_0 \in (0,1)$. Then
for $\nu^+$-almost all $\omega$,
\[ \lim_{n\to \infty} \left| f^n_\omega (x_0) - f^n_\omega (y_0) \right| =0.\]
\end{theorem}

\begin{proof}
The proof of Theorem~\ref{t:kan} gives the existence of
 an invariant measurable graph $\xi: \Sigma_2 \to \mathbb{I}$
 so that given any $x_0 \in (0,1)$, one has that for $\nu$-almost all $\omega$,
\begin{align}\label{e:xi}
 \lim_{n\to\infty} \left| f^n_{\sigma^{-n} \omega} (x_0) - \xi(\omega) \right| &= 0.
\end{align}
As $\nu$ is invariant for $\sigma$, this proves that $f^n_\omega (x_0)$ converges to $\xi(\sigma^n \omega)$ in probability. 
This gives the existence of a subsequence $n_k \to \infty$ as $k\to \infty$ so that
\[ 
\lim_{k\to\infty} \left| f^{n_k}_\omega (x_0) - \xi(\sigma^{n_k} \omega) \right| =0
\]
(see e.g. \cite[Theorem~II.10.5]{shi84}).
We have thus obtained the weaker statement that
for $\nu^+$-almost all $\omega$,
\[
 \liminf_{n\to \infty} \left| f^n_\omega (x_0) - \xi(\sigma^{n} \omega)\right| =0.
\]

We provide  a sketch of the argument to show that this is also true with  limit replacing the limit inferior,
which would imply the theorem.
The measure $(\text{id} , \xi) \nu$ on $\Sigma_2 \times \mathbb{I}$, with conditional measures $\delta_{\xi(\omega)}$ and marginal $\nu$, 
is invariant 
for the natural extension $F: \Sigma_2 \times \mathbb{I} \to \Sigma_2 \times \mathbb{I}$
of $F^+$. It corresponds to an invariant measure $\nu^+ \times m$ for $F^+$ by Proposition~\ref{p:fapp} and the observation that
$\xi (\omega)$ depends on $(\omega_n)_{-\infty}^{-1}$ only.

\begin{lemma}\label{l:hyp}
With  respect to the measure $\nu^+ \times m$, the system $F^+$ has a negative Lyapunov exponent;
\[
\lambda = \sum_{i=1}^2 p_i \int_{\mathbb{I}} \ln (f_i'(x)) \, dm(x) <0.
\]
\end{lemma}

\begin{proof}
One can follow the argument for \cite[Theorem~7.1]{klevol14} (an analogue of \cite[Theorem~4.2]{bax89}). We describe the steps.
A key idea is the use of the notion of relative entropy; the relative entropy $h(m_1 \vert m_2)$ of a probability measure $m_1$ on $\mathbb{I}$ with respect to a
probability measure $m_2$  on $\mathbb{I}$  is given by
\[
h(m_1 \vert m_2 ) = \sup_{\psi \in C^0(\mathbb{I})} \left( \ln \left( \int_{\mathbb{I}} e^{\psi(x)} \, dm_1 (x) \right) - \int_{\mathbb{I}} \psi(x) \, dm_2(x)   \right).
\]
The following properties hold \cite{donvar75}.
\begin{enumerate}[label=(\roman*)]
\item  \label{i:re1}
$0 \le h(m_1 \vert m_2) \le \infty$;
\item \label{i:re2}
$h( m_1 \vert m_2) = 0$ if and only if $m_1 = m_2$;
\end{enumerate}
%

A relation between  Lyapunov exponent and relative entropy can be derived for absolutely continuous stationary measures.
The argument now involves maps with absolutely continuous noise with shrinking amplitude to approximate the fiber diffeomorphisms.
Such a perturbed system admits an absolutely continuous stationary measure. One uses the relation 
 between the Lyapunov exponent and relative entropy for this absolutely continuous stationary measure and considers the limit
 where the noise amplitude shrinks to zero.

Let $\zeta$ be a random variable with values in $[0,1]$ that is uniformly distributed.
For small positive values of $\varepsilon$,
let \[f_{i,\zeta} (x) = (1-\varepsilon)f_i(x) + \zeta \varepsilon.\]
Note that \[\zeta \varepsilon = f_{i,\zeta} (0) \le f_{i,\zeta} (x) \le f_{i,\zeta} (1) = 1 - \varepsilon + \zeta \varepsilon,\]
so that  $f_{i,\zeta}$ maps $[0,1]$ into $[0,1]$, for each value of $\zeta$ in $[0,1]$.
The iterated function system generated by the maps $f_{i,\zeta}$  has a stationary measure
\begin{align}\label{e:Te}
m_\varepsilon &= \sum_{i=1}^2 p_i \int_{0}^{1} f_{i,\zeta} m_\varepsilon \, d\zeta.
\end{align}
Note that $m_\varepsilon$ is a fixed point of
an operator $\mathcal{T}_\varepsilon$ where $\mathcal{T}_\varepsilon m_\varepsilon$ is defined by the right hand side of \eqref{e:Te}.
One can show that $m_\varepsilon$ has a smooth density \cite{zmahom07}.
Moreover, with $\mathcal{N}_c$  the closed set
of probability measures considered in the proof of Lemma~\ref{l:dmeas}, one has that for
for suitable values of $\alpha,c,q$,
$m_\varepsilon \in \mathcal{N}_c$ for all small positive $\varepsilon$.
This is true since $\mathcal{T}_\varepsilon $ maps $\mathcal{N}_c$ into itself for suitable values of $\alpha,c,q$.
To see this follow the proof 
of Lemma~\ref{l:dmeas} with $\mathcal{T}_\varepsilon$ replacing $\mathcal{T}$.
The main calculation analogous to \eqref{e:Tinside} 
is straightforward noting  
that $f_{i,\zeta}^{-1} ([0,x)) \subset f_{i,0}^{-1} ([0,x))$ for all $\zeta \in [0,1]$:
\begin{align*}
\mathcal{T}_\varepsilon  m_\varepsilon  \big([0,x)\big) 
&=
\sum_{i=1}^2  p_{i} \int_0^1 f_{i,\zeta} m_\varepsilon \big([0,x)\big) \, d\zeta
=
\sum_{i=1}^2  p_{i} \int_0^1 m_\varepsilon\left(f^{-1}_{i,\zeta}[0,x)\right) \, d\zeta 
\\
&\le \sum_{i=1}^2  p_{i} \int_0^1 m_\varepsilon\left(f^{-1}_{i,0}[0,x)\right) \, d\zeta
\leq
\sum_{i=1}^2  p_{i} m_\varepsilon  \left( [0 , \frac{x}{\rho_{i}-\delta}) \right) 
\\
&\leq
\sum_{i=1}^{2} p_i c \left(\frac{x}{\rho_i-\delta}\right)^\alpha 
=
c \left(\sum_{i=1}^{2} \frac{p_i}{(\rho_i-\delta)^\alpha}\right) x^\alpha 
\leq
cx^\alpha. 
\end{align*}
A similar argument can be employed near the  boundary point $1$, for $\varepsilon$ small.

The Lyapunov exponent for the stationary measure $m_\varepsilon$ is given by
\begin{align}\label{e:lyap}
\lambda_\varepsilon &= \sum_{i=1}^2 p_i \int_0^1 \int_{\mathbb{I}}  \ln(f_{i,\zeta}' (x)) \, dm_\varepsilon (x) d\zeta.
\end{align}
Since $m_\varepsilon$ has a smooth and bounded density, one can prove the relation (see \cite[Proposition~7.2]{klevol14})
\begin{align}\label{e:72}
\lambda_\varepsilon &= -\sum_{i=1}^2 p_i \int_0^1 h ( f_{i,\zeta} m_\varepsilon \vert m_\varepsilon  ) \, d\zeta.
\end{align}
By \ref{i:re1}, $\lambda_\varepsilon \le 0$. 
By \ref{i:re2}, $\lambda_\varepsilon = 0$ if and only if $f_{i,\zeta} m_\varepsilon = m_\varepsilon$ 
for all $\zeta \in [0,1]$ and $i=1,2$. As the latter is not possible,  $\lambda_\varepsilon <0$.

Now take the limit $\varepsilon \to 0$.
Then $m_\varepsilon \to m$ since $\mathcal{T}_\varepsilon$ is continuous and depends continuously on $\varepsilon$ (compare Lemma~\ref{l:c0}), convergence is in $\mathcal{N}_c$, and $m$ is the unique stationary measure  in $\mathcal{N}_c$.
From \eqref{e:lyap} one sees that $\lambda_\varepsilon \to \lambda$ as $\varepsilon \to 0$ and 
we obtain $\lambda \le 0$. 
Since the relative entropy is lower 
semi-continuous in $\varepsilon$ as the supremum of continuous functionals, one finds from \eqref{e:72} that 
\[
0 \le \sum_{i=1}^2 p_i h ( f_{i} m \vert m  )  \le  -\lambda.
\]
This shows that $h ( f_i m \vert m) = 0$ for $i=1,2$ in case $\lambda = 0$. This is clearly not the case by \eqref{i:re2}, 
as $f_1 m \ne f_2 m \ne m$. So we have $\lambda <0$.
\end{proof}

Because of this lemma, for $\nu$-almost all $\omega$, $\xi (\omega)$ from \eqref{e:xi} 
has a stable manifold $W^s (\omega)$ that is an open neighborhood
of $\xi(\omega)$ in $\mathbb{I}$. To see this one can refer to  general theory for nonuniformly hyperbolic systems as in \cite{barpes07}, or
apply reasoning as in Lemma~\ref{l:bonmil}.
For each $x\in W^s (\omega)$, 
\begin{align*}
\lim_{n\to \infty} \left| f^n_\omega (x) - \xi (\sigma^n \omega) \right| = 0.
\end{align*}
Write 
\[
W^s (\omega) = (r^b (\omega) , r^t(\omega)).
\]
%
Then $r^b$ and $r^t$ are invariant.
 Hence $r^b>0$, $\nu$-almost everywhere, or $r^b=0$, $\nu$-almost everywhere, and likewise
 $r^t<1$, $\nu$-almost everywhere, or $r^t=1$, $\nu$-almost everywhere.
 We will derive a contradiction from the assumption that $r^t <1$ or $r^b >0$, $\nu$-almost everywhere.
Assume that e.g. $r^t <1$, $\nu$-almost everywhere.
Write
\begin{align}\label{e:ris}
r(\omega) &= \inf \{ x\in \mathbb{I} \; \mid \; \lim_{n\to \infty}  f^{-n}_\omega (x) =1 \}.
\end{align}
As $L(1) > 0$, we have $r (\omega) <1$ for $\nu$-almost all $\omega \in \Sigma_2$,
compare Lemma~\ref{l:bonmil}.
Since the graphs of $r^t$ and $r$ are invariant graphs and also $r^t <1$, we have $r \ge r^t > \xi$, $\nu$-almost everywhere. 

The measure $\mu = (\text{id},r)\nu$ on $\Sigma_2 \times \mathbb{I}$ with conditional measures $\delta_{r(\omega)}$ and marginal $\nu$
defines an invariant measure for $F$.
It follows from the expression \eqref{e:ris} that  $r(\omega)$ depends on the past $\omega^- = (\omega_n)_{-\infty}^{-1}$ only.
Consequently, $\mu$ is a product measure of the form $\nu^+ \times \vartheta$ on $\Sigma^+_2 \times (\Sigma^-_2 \times \mathbb{I})$.
With $\Pi$ the natural projection $\Sigma_2 \times \mathbb{I} \to \Sigma_2^+ \times \mathbb{I}$,
we find that the $F^+$-invariant measure
 $\Pi \mu$ is a product measure 
 $\Pi \mu = \nu^+ \times \hat{m}$ on $\Sigma_2^+ \times \mathbb{I}$. 
By Lemma~\ref{l:corres}, $\hat{m}$ is a stationary measure.
Since $r >\xi$, $\nu$-almost everywhere, Proposition~\ref{p:fapp} gives that $m \ne \hat{m}$.
Lemma~\ref{l:unique} however prohibits the existence of two different stationary measures with support in $(0,1)$.
The contradiction is derived, establishing that
$W^s (\omega) = (0,1)$ for $\nu$-almost all $\omega\in \Sigma_2$.
\end{proof}

 Under the conditions of Theorem~\ref{t:sync}, the proof of Theorem~\ref{t:kan} shows that 
 \[
 \lim_{n\to \infty} \left| f^n_{\sigma^{-n} \omega} (x) - \xi (\omega) \right| = 0
 \]
 for $\nu$-almost all $\omega \in \Sigma_2$ and any $x \in (0,1)$. This convergence is called pullback convergence. 
 The proof of Theorem~\ref{t:sync} shows that
 \[
  \lim_{n\to \infty} \left|  f^n_{\omega} (x) - \xi( \sigma^n \omega) \right| = 0
  \]
  for $\nu$-almost all $\omega \in \Sigma_2$ and any $x \in (0,1)$.
 This convergence is called forward convergence.
 So in this case both pullback and forward convergence to $\xi$ holds.
 In general however forward convergence is not a consequence of pullback convergence.
 The next section provides an example, involving a zero Lyapunov exponent, with pullback convergence but not forward convergence.
 See in particular Section~\ref{ss:pullback}.
 Section~\ref{s:riddled} contains a related example, related by going to the inverse skew product system, with
 forward convergence but not pullback convergence.
 We refer to \cite{kloras11} for more discussion on conditions for convergence in nonautonomous and skew product systems.
 
 We finish with some pointers to further literature.
 In \citep{ant84,klenal04,derklenav07,zmahom08}  synchronization results, similar to Theorem~\ref{t:sync}, for skew product systems with circle diffeomorphisms
 as fiber maps are treated without employing negativity of Lyapunov exponents. 
 Motivated by Lemma~\ref{l:hyp} for example, 
 one may wonder about other invariant measures than those with Bernoulli measure as marginal.
 Reference \cite{gorilyklenal05} considers, in this direction, the existence of nonhyperbolic measures for step skew product systems with circle fibers.

\section{On-off intermittency}\label{s:on-off}

Intermittency in a dynamical system stands for dynamics that exhibits alternating phases of different characteristics.
Typically, intermittent dynamics alternates time series close to equilibrium with bursts of global dynamics \cite{berpomvid86}.
In our context,  we say that a step skew product system ${F^+} \in \mathcal{S}$ 
displays intermittency if the following holds
for any sufficiently small neighborhood $U$ of $0$: 
\begin{enumerate}
 \item
 For all $x \in (0,1)$ and $\nu^+$-almost all $\omega \in \Sigma_2^+$:
\[
 \lim_{n\to \infty} \frac{1}{n} \left\vert   \{ 0\le i< n \; ; \; f^i_\omega (x) \in U   \} \right\vert = 1; 
\]
\item For all $x \in (0,1)$ and $\nu^+$-almost all $\omega \in \Sigma_2^+$,
$f^n_\omega (x) \not \in U$ for infinitely many $n$.
\end{enumerate}
Here, for a finite set $S$, we write $\left\vert S \right\vert$ for its cardinality.

This kind of intermittency that involves a weakly unstable invariant set, here $\Sigma_2^+ \times \{0\} \subset \Sigma_2^+ \times \mathbb{I}$,
has been called on-off intermittency \citep{plsptr93,heaplaham94}.
The occurrence of intermittency in iterated function systems
of logistic maps with zero Lyapunov exponent at the fixed point in $0$ is treated in \citep{atda00,atsc03}.
 See also  \cite{bonmil08} for a study of specific interval diffeomorphisms
over expanding circle maps.
\begin{figure}[htbp]

\begin{center}
 \includegraphics[height=5.2cm]{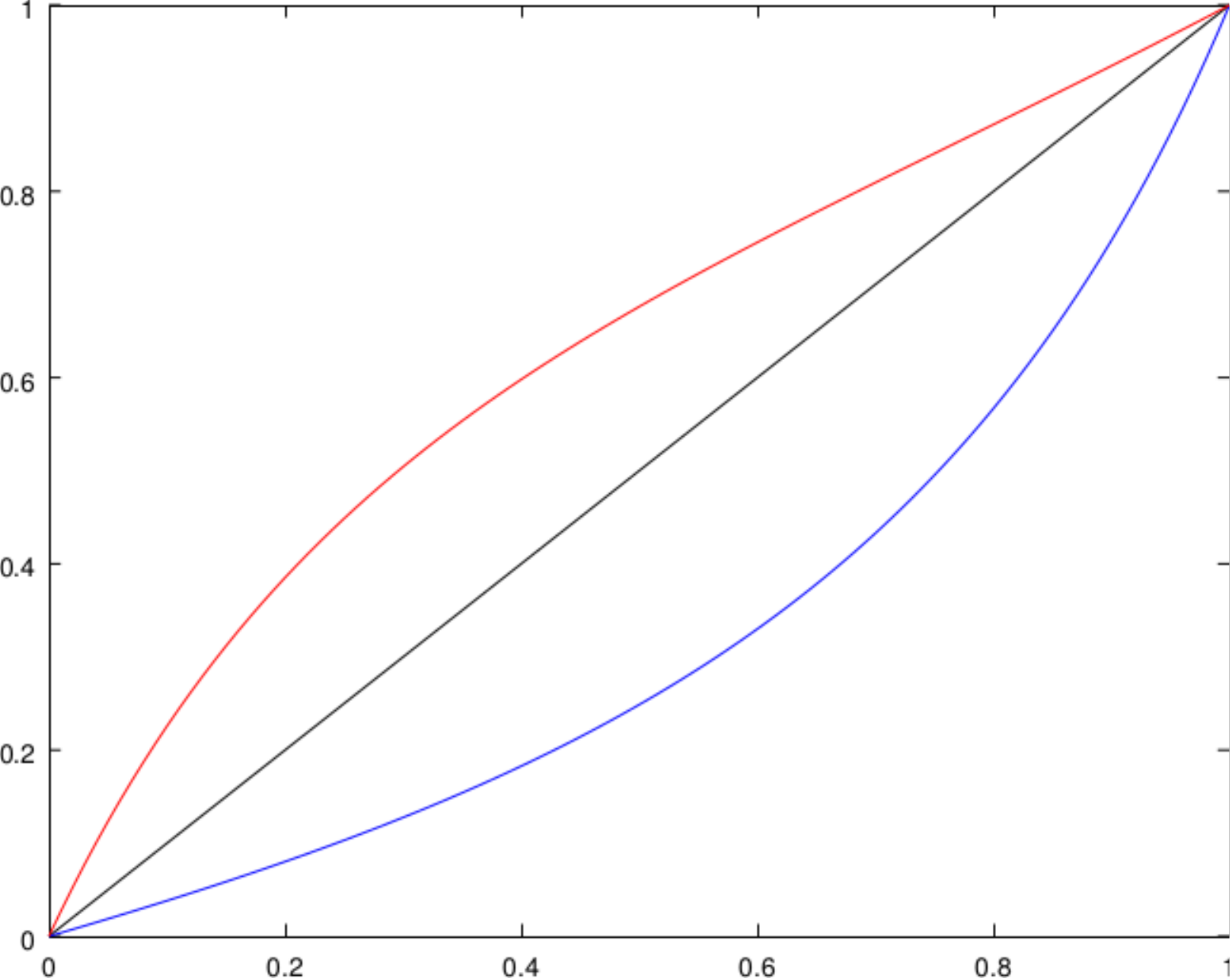}\hspace{0.4cm} 
 \includegraphics[height=5.2cm]{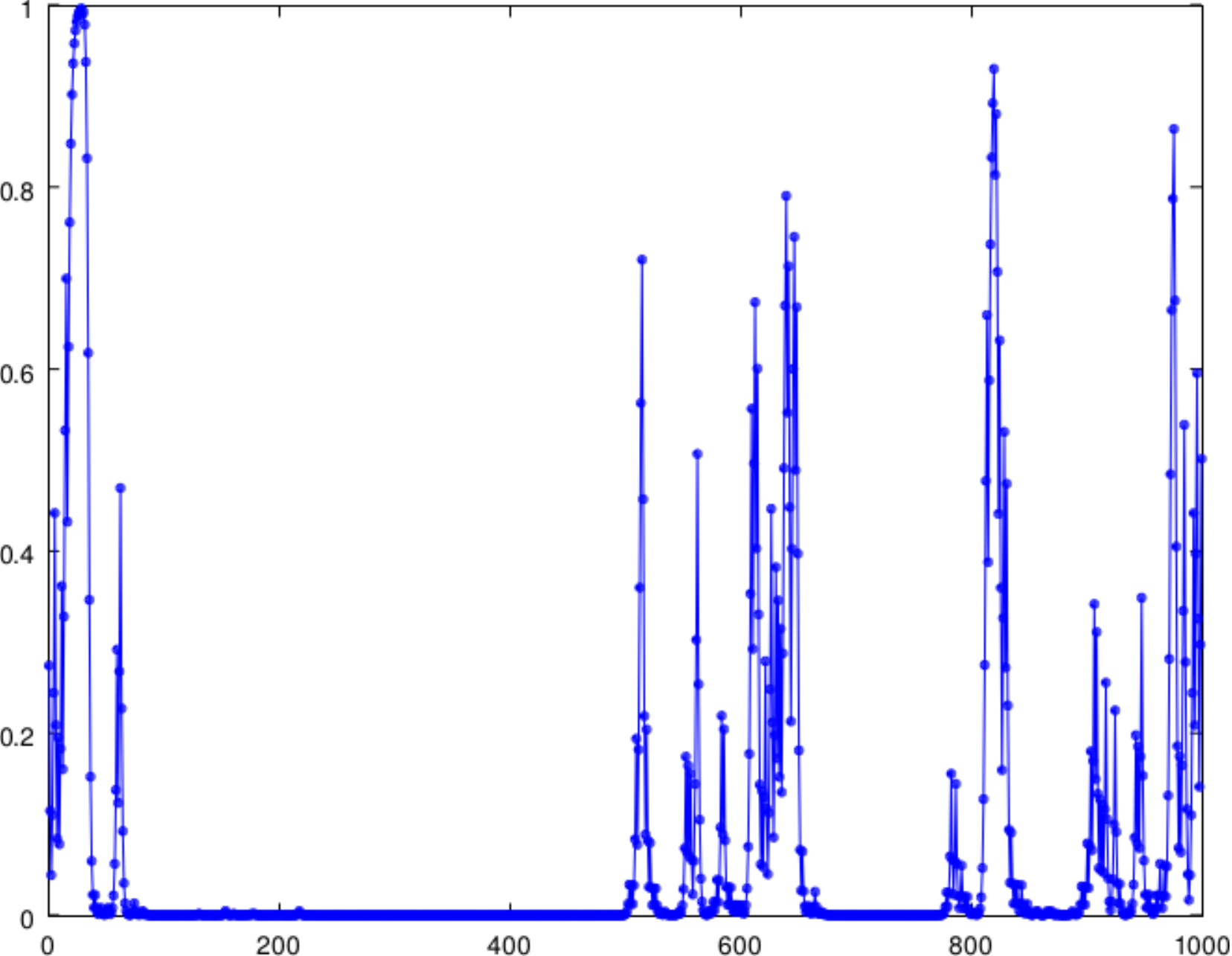} 
\caption{The left frame depicts the graphs of $x \mapsto f_i(x) = g_i(x) ( 1 - p(x))$, $i=1,2$, with $g_i (x)$ as in \eqref{e:g1}, \eqref{e:g2} and $p(x) = \frac{3}{10}  x (1-x)$. 
The corresponding step skew product system has a zero Lyapunov exponent along $\Sigma_2^+ \times \{0\}$ and a positive Lyapunov exponent along
$\Sigma_2^+ \times \{1\}$.
The right frame shows a time series for the iterated function system generated by these diffeomorphisms.
\label{f:onoff}}
\end{center}

\end{figure}

In this section we will discuss on-off intermittency for 
step skew product systems ${F^+} \in \mathcal{S}$.
Throughout we assume that both diffeomorphisms $f_1, f_2$ are picked with probability $1/2$. 
This is for convenience, we expect the more general case with probabilities $p_1, p_2$ to go along the same lines. 
The following two theorems, Theorems~\ref{t:beta} and \ref{t:delta0}, demonstrate that ${F^+} \in \mathcal{S}$ with $L(0)=0$ and $L(1)>0$ displays intermittency. 
Figure~\ref{f:onoff} illustrates a typical time series.

Lamperti, in a sequence of papers \citep{lam60,lam62,lam63}, developed a general theory of recurrence for nonhomogeneous
random walks on the half-line. 
His results may be used to prove on-off intermittency in our context, see in particular \cite[Theorems 3.1 and 4.1]{lam63}. 
We will get it by calculating bounds on stopping times, using $C^2$ differentiability of the generating diffeomorphisms.

 \begin{theorem}\label{t:beta}
 Let ${F^+} \in \mathcal{S}$ and assume $L(0) = 0$.
 Let $0< \beta$ be small and $x_0 \in (0,1)$.
  Then for $\nu^+$-almost every $\omega \in \Sigma_2^+$,
  $f^n_\omega (x_0)$ is in $[\beta , 1]$ for infinitely many values of $n$.  
\end{theorem}

\begin{proof}
We follow the proof of \cite[Theorem~1]{atsc03}.
Given $x_0 \in \mathbb{I}$ and $\omega \in \Sigma_2^+$, write \[x_n = x_n (\omega) = f^n_\omega (x_0).\]
It suffices to show that for $x_0 < \beta$,
\[
 \nu^+ ( \{ \omega \in \Sigma_2^+ \; \mid \; x_n \ge \beta  \text{ for some } n\ge 1 \}) =1.
\]
Let \[u_n  = - \ln (x_n).\]
We wish to show that for $u_0 > K = -\ln (\beta)$,
\[
 \nu^+ ( \{ \omega \in \Sigma_2^+ \; \mid \; u_n \le K  \text{ for some } n\ge 1 \}) =1.
\]

Write $f_i(x) = c_i x / (1 + t_i(x))$ with $t_i (x) = O(x)$ as $x \to 0$. 
Taking logarithms of $x_{n+1} = f_{\omega_n} (x_n)$ we get 
\[
u_{n+1}  = d_{\omega_n} + u_n + \ln ( 1 + t_{\omega_n}( e^{-u_n})  ),
\]
with $d_i = -\ln (c_i)$.
Consider the stopping time 
\[
T = \inf \{ n\ge 1 \mid u_n \le K  \}
\]
 (with $T = \infty$ if $u_n > K$ for all $n$)
and write \[z_n = \ln (u_{n \wedge T} ),\]
where $n \wedge T = \min \{ n,T\}$.
We claim that $z_n$ is a supermartingale;

\begin{lemma}\label{l:supermario}
\[
 \int_{C}  z_{n+1} (\omega) \, d\nu^+(\omega) \le \int_{C}  z_{n} (\omega) \, d\nu^+(\omega)
\]
for 
cylinders $C = C^{0,\ldots,n-1}_{\omega_0,\ldots,\omega_{n-1}}$.
\end{lemma}

\begin{proof}
On $C$, $z_n(\omega)$ is constant.
As further $c_{\omega_{n+1}}$ is independent of $c_{\omega_n}$, it suffices to consider $n=0$ and to prove
\[
 \int_{\Sigma_2^+} z_1 (\omega) \, d\nu^+ (\omega) \le \ln (u_0)
\]
for $u_0$ large enough.
Denote
\[
h(u_0) = \int_{\Sigma_2^+} z_{1}\, d\nu^+(\omega) - \ln (u_0).
\]
The zero Lyapunov exponent,  $L(0) = 0$, implies  
$\int_{\Sigma_2^+} d_{\omega_0} \, d\nu^+(\omega) = 0$.
Using this,
\begin{align*}
h(u) &=
 \int_{\Sigma_2^+} \ln \left(d_{\omega_0} + u + \ln(1 + t_{\omega_0}( e^{-u})) \right) - \ln (u) \, d\nu^+(\omega)
\\
&= \int_{\Sigma_2^+} \ln \left( \frac{d_{\omega_0} + u + \ln(1 + t_{\omega_0}( e^{-u}))}{u}\right) \, d\nu^+(\omega)
\\
&= \int_{\Sigma_2^+}  \ln \left( \frac{d_{\omega_0} + u + \ln(1 + t_{\omega_0}( e^{-u}))}{u}\right)  - \frac{d_{\omega_0}}{u} \, d\nu^+(\omega)
 \\
 &= 
 \int_{\Sigma_2^+} \ln \left(\frac{d_{\omega_0}}{u} + 1 + \frac{\ln(1 + t_{\omega_0}( e^{-u}))}{u} \right) - \frac{d_{\omega_0}}{u}  \, d\nu^+(\omega).
 \end{align*}
 By developing the integrand of the last expression in a Taylor expansion, this gives
 \begin{align*}
 h(u)
 &=
 -\frac{1}{2} \int_{\Sigma_2^+}  \left( \frac{d_{\omega_0}}{u} \right)^2 + o \left(\frac{1}{u^2}\right)   \, d\nu^+(\omega), \, u \to \infty. 
\end{align*}
So
\[
\limsup_{u\to\infty} h(u) u^2 < 0,   
\]
implying that there exists $\bar{u}$ so that $h(u)<0$  for $u \ge \bar{u}$.
\end{proof}

Now that Lemma~\ref{l:supermario} gives that 
$z_n$ is a nonnegative supermartingale, by Doob's supermartingale convergence theorem, see e.g. \cite[Section~VII.4]{shi84},
\begin{align}\label{e:b1=1} \lim_{n\to\infty} z_n(\omega) &< \infty \end{align}
for $\nu^+$-almost all $\omega \in \Sigma^+_2$. 
Let 
\begin{enumerate}
 \item $B_1 = \{ \omega\in\Sigma_2^+ \mid z_\infty = \lim\limits_{n\to \infty} z_n < \infty\}$;
 \item $B_2 = \{ \omega\in\Sigma_2^+ \mid T = \infty\}$.
\end{enumerate}
We must prove that $\nu^+ (B_2) = 0$.
On $B_1 \cap B_2$, $z_n \to z_\infty$ and thus $x_n \to x_\infty \in (0,1)$ as $n \to \infty$.  
This is impossible as both $f_1 (x_\infty) \ne x_\infty$ and $f_2(x_\infty) \ne x_\infty$.
So $B_1 \cap B_2 = \emptyset$.
By \eqref{e:b1=1}, $\nu^+ (B_1) = 1$. Hence $\nu^+ (B_2) = 0$. 
\end{proof}

If one assumes $L(1) >0$, then   
a similar, in fact simpler, argument shows that for $\beta$ small and $x_0 \in (0,1)$, for $\nu^+$-almost all $\omega$ one finds  
$x_n = f^n_\omega (x_0)$ in $[0 , 1-\beta]$ for infinitely many values of $n$.

\begin{theorem}\label{t:delta0}
Consider ${F^+} \in \mathcal{S}$ and assume $L(0) = 0$ and $L(1) >0$.
 Let $0< \beta < 1$ and $x_0 \in (0,1)$.
 Then for $\nu^+$-almost every $\omega \in \Sigma_2^+$,
 \begin{align} \label{e:d0}
  \lim_{n\to \infty} \frac{1}{n} \sum_{i=0}^{n-1} \mathbbm{1}_{[0,\beta)} (f^i_\omega (x_0)) &= 1.
 \end{align}
\end{theorem}

\begin{proof}
The reasoning is inspired by \citep[Theorem~4]{atsc03}.
Consider \[x_n = x_n(\omega) = f_\omega^n (x_0)\] and \[y_n = \ln (x_n / (1 - x_n)).\] 
We denote
\begin{align*} 
y_{n+1} &= h_{\omega_{n}} (y_n).
\end{align*} 
For $\beta$ small,  $K = \ln (\beta / (1-\beta))$ is a large negative number. 
For definiteness assume $x_0 \le \beta$, i.e. $y_0 \le K$.
Define stopping times $T_0 = 0$,
\begin{align*}
T_{2k+1} &= \inf \{ n \in \mathbb{N} \;  \mid  \; n > T_{2k} \text{ and } y_n > K\}, 
\\
T_{2k} &= \inf \{ n \in \mathbb{N} \; \mid \; n > T_{2k-1} \text{ and } y_n \le K\},
\end{align*}
see Figure~\ref{f:stop}.
Let 
\begin{align*}
\eta_k &= \left| [T_{2k-2}, T_{2k-1}) \right| = T_{2k-1} - T_{2k-2}, \\
\xi_k &= \left| [T_{2k-1}, T_{2k}) \right| = T_{2k} - T_{2k-1} 
\end{align*}
be the duration of subsequent iterates with $y_n \le K$ 
and the duration of subsequent iterates with $y_n > K$, 
respectively.
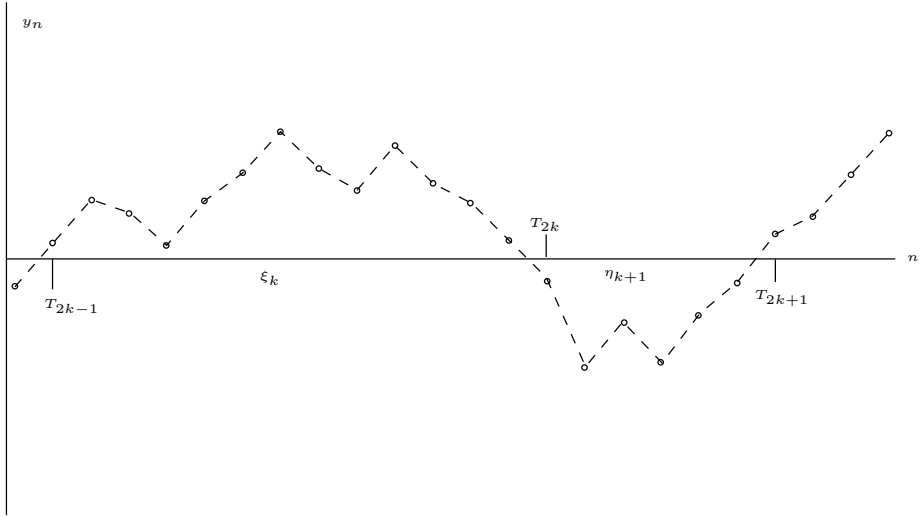
\begin{figure}[h]
\begin{picture}(0,0)%
\includegraphics{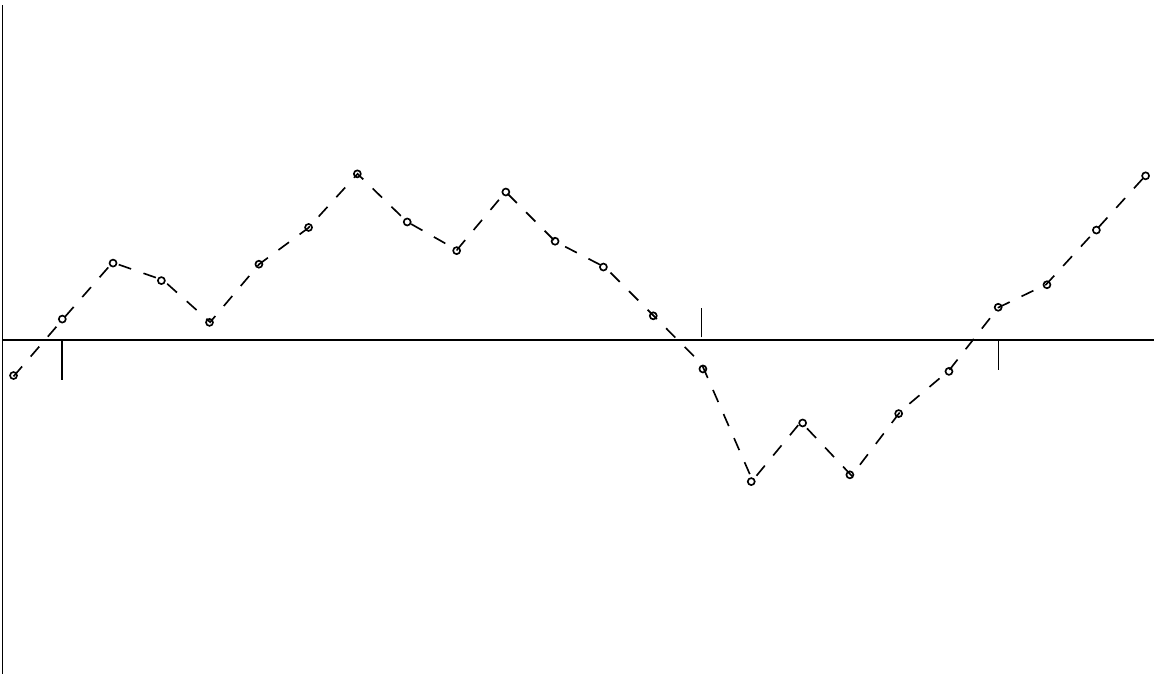}%
\end{picture}%
\setlength{\unitlength}{4144sp}%
\begingroup\makeatletter\ifx\SetFigFont\undefined%
\gdef\SetFigFont#1#2#3#4#5{%
  \reset@font\fontsize{#1}{#2pt}%
  \fontfamily{#3}\fontseries{#4}\fontshape{#5}%
  \selectfont}%
\fi\endgroup%
\begin{picture}(5359,3084)(393,-2683)
\put(496,254){\makebox(0,0)[lb]{\smash{{\SetFigFont{5}{6.0}{\familydefault}{\mddefault}{\updefault}{\color[rgb]{0,0,0}$y_n$}%
}}}}
\put(1905,-1268){\makebox(0,0)[lb]{\smash{{\SetFigFont{5}{6.0}{\familydefault}{\mddefault}{\updefault}{\color[rgb]{0,0,0}$\xi_{k}$}%
}}}}
\put(3941,-1244){\makebox(0,0)[lb]{\smash{{\SetFigFont{5}{6.0}{\familydefault}{\mddefault}{\updefault}{\color[rgb]{0,0,0}$\eta_{k+1}$}%
}}}}
\put(629,-1436){\makebox(0,0)[lb]{\smash{{\SetFigFont{5}{6.0}{\familydefault}{\mddefault}{\updefault}{\color[rgb]{0,0,0}$T_{2k-1}$}%
}}}}
\put(3501,-949){\makebox(0,0)[lb]{\smash{{\SetFigFont{5}{6.0}{\familydefault}{\mddefault}{\updefault}{\color[rgb]{0,0,0}$T_{2k}$}%
}}}}
\put(4836,-1381){\makebox(0,0)[lb]{\smash{{\SetFigFont{5}{6.0}{\familydefault}{\mddefault}{\updefault}{\color[rgb]{0,0,0}$T_{2k+1}$}%
}}}}
\put(5737,-1156){\makebox(0,0)[lb]{\smash{{\SetFigFont{5}{6.0}{\familydefault}{\mddefault}{\updefault}{\color[rgb]{0,0,0}$n$}%
}}}}
\end{picture}%
\caption[]{A sequence of stopping times is defined to label subsequent iterates where $y_n$ leaves $(-\infty,K]$ or $(K,\infty)$. \label{f:stop}}
\end{figure}

Lemmas~\ref{l:Tinfty} and \ref{l:Tfinite} 
determine bounds for the expectation of the stopping times 
$\eta_k$ (which is shown to be infinite) and
$\xi_k$ 
(which is shown to be finite).

\begin{lemma} \label{l:Tinfty}
\begin{align} \label{e:waldrho}
 \int_{\Sigma_2^+} \eta_k (\omega) \, d \nu^+ (\omega) = \infty;
\end{align}
 $\eta_k$ has infinite expectation.
\end{lemma}

\begin{proof}
A smooth conjugation brings $f_1$ near $0$ to the linear map $x \mapsto x/d$ for some $d>1$. More formally, there is a local diffeomorphism
$h: U \to \mathbb{R}$ with $U$ a neighborhood of $0$ and $h(0)=0$, so that $h^{-1} \circ f_1 \circ h (x) = x/d$. 
Replace $f_1$ by $h^{-1} \circ f_1 \circ h$ and likewise $f_2$ by $h^{-1} \circ f_2 \circ h$, near $0$.
That is, we may assume that on a neighborhood of $0$ that contains $[0,\beta]$,
\begin{align*}
f_1 (x) &= x/d,\\
f_2 (x) &= d x (1 + r(x)),
\end{align*} 
for some smooth function $r(x) = O(|x|)$, $x \to 0$ (observe that  $f_1' (0)  = 1/ f_2'(0)$ by $L(0)=0$).

Let $\tilde{z}_n = \ln (x_n)$ map $(0,\beta]$ to $(-\infty , \tilde{L}]$ with $\tilde{L} = \ln (\beta)$.
Then $x_{n+1} = f_{\omega_n} (x_n)$ becomes
\begin{align*}
\tilde{z}_{n+1} &= \left\{ \begin{array}{ll} \tilde{z}_n - \ln(d), & \text{ if } \omega_n=1, \\ \tilde{z}_n + \ln(d) + \ln (1 + r(e^{\tilde{z}_n})), & \text{ if } \omega_n=2. \end{array}\right.   
\end{align*}
An additional rescaling $z_n = \tilde{z}_n/\ln(d)$  conjugates this iterated function to
\begin{align}
\label{e:yn}
z_{n+1} &= \left\{ \begin{array}{ll} z_n - 1, & \text{ if } \omega_n=1, \\ z_n + 1 + \ln (1 + r(e^{z_n \ln(d)}))/\ln(d), & \text{ if } \omega_n=2. \end{array}\right.   
\end{align} 
Write $L = \tilde{L}/\ln(d)$; we consider $z_n$ on $(-\infty,L]$. 

Let $g>0$; $g$ will be chosen large in the sequel.
The term $\ln (1 + r(e^{z_n\ln(d)}))/\ln(d)$ may be bounded from above by $C e^{-g\ln(d)}$ on intervals $(-\infty , L - g]$, for some $C>0$.
On $(-\infty,L - g] \subset (-\infty,L]$ we compare the random walk $z_n$ with the random walk 
\begin{align*}
v_{n+1} &= \left\{ \begin{array}{ll} v_n - 1, & \text{ if } \omega_n=1, \\ v_n + 1 + C e^{-g \ln(d)}, & \text{ if } \omega_n=2. \end{array}\right.   
\end{align*}
Given $z_0 = v_0 \in [L-g-1, L - g)$, we define stopping times 
\begin{align*} 
T_z &= \min \{n \in \mathbb{N} \; \mid \; z_n \ge L - g\}, 
\\
T_v &= \min \{ n\in\mathbb{N}\; \mid \;v_n \ge L - g\}.
\end{align*}
If $z_0,\ldots,z_n \in (-\infty,L - g)$, then $z_i \le v_i$ for all $0\le i \le n+1$.
Therefore, for each $\omega \in \Sigma_2^+$, 
\[
T_z (\omega) \ge T_v (\omega).
\]
By Wald's identity, see e.g. \cite[Section~VII.2]{shi84},
\begin{align*}
\int_{\Sigma_2^+} T_v (\omega) \, d\nu^+ (\omega) &=  \frac{2}{ C e^{-g \ln (d)}} \int_{\Sigma_2^+} v_{T_v (\omega)} - v_0 \, d\nu^+ (\omega)
\\
&\ge c e^{g \ln(d)}
\end{align*}
for some $c>0$, and hence  
%
%
%
\begin{align}\label{e:Ty}
\int_{\Sigma_2^+}  T_z (\omega) \, d\nu^+ (\omega) &\ge c e^{g \ln(d)}.
\end{align}

Let $\alpha>0$ be so that 
\begin{align}\label{e:yalpha}
z_{n+1} &\le z_n + 1 + \alpha
\end{align} 
for $L - g \le z_n \le L$.
Note that we may take $\alpha$ to be small if $L$ is large.
Consider the random walk given by \eqref{e:yn} with initial point $z_0 \in (L-1,L]$.
Define the stopping time 
\begin{align*}
T_g &= \min\{ n >0 \;  \mid  \;  z_n < L-g \text { or } z_n > L \}.
\end{align*}

\begin{lemma}\label{l:AB}
For $\alpha>0$ small enough, there is  $c, r^* = r^*(\alpha)<0$, so that
 \[
 \nu^+ (\{ \omega \in \Sigma_2^+ \;  \mid  \; z_{T_g} < L - g \}) \ge c e^{g r^*}.
 \]
Here $r^*(\alpha) \to 0$ as $\alpha \to 0$.   
\end{lemma}

We finish the proof of Lemma~\ref{l:Tinfty} using this lemma, and then prove Lemma~\ref{l:AB}.
Consider the following reasoning.
Start with a point $z_{T_{2k}} \in (L-1,L]$.
Then some iterate of $z_{T_{2k}}$ will have left $[L - g,L]$, either through the right boundary point $L$ or, with  probability
determined by Lemma~\ref{l:AB}, 
through the left boundary point $L - g$. 
In the latter case there will be a return time to $[L-g,L]$ after which a further iterate may leave through the right boundary point $L$.
Consequently, combining \eqref{e:Ty} and Lemma~\ref{l:AB},
\begin{align}\label{e:compare}
\int_{\Sigma_2^+}  \eta_k (\omega) \, d\nu^+ (\omega) &\ge  c e^{g r^*}   e^{g \ln (d)}
\end{align}
for some $c>0$.
For $L$ sufficiently large, $\alpha$ is small enough to ensure
$ e^{r^*}   e^{\ln (d)} >1$, because $r^*(\alpha) \to 0$ as $\alpha \to 0$.
Then the right hand side of \eqref{e:compare} goes to infinity as $g \to \infty$.
This concludes the proof of Lemma~\ref{l:Tinfty}.
\end{proof}

\begin{proof}[Proof of Lemma~\ref{l:AB}]
Consider the random walk
\begin{align*}
u_{n+1} &= \left\{ \begin{array}{ll} u_n - 1, & \text{ if } \omega_n=1, \\ u_n + 1 + \alpha, & \text{ if } \omega_n=2, \end{array}\right.   
\end{align*}
with $ u_0 \in (-1,0]$.
Define the stopping time 
\begin{align*}
U_g &= \min\{ n >0 \;  \mid  \;  u_n < -g \text { or } u_n >0 \}.
\end{align*}
By \eqref{e:yalpha}
we have
 \begin{align*} 
 \nu^+ (\{ \omega \in \Sigma_2^+ \;  \mid  \; z_{T_g} < L-g \}) &\ge \nu^+ (\{ \omega \in \Sigma_2^+ \;  \mid  \; u_{U_g} < -g \})
 \end{align*}
and hence
 it suffices to prove the estimate
 \begin{align*} 
 \nu^+ (\{ \omega \in \Sigma_2^+ \;  \mid  \; u_{U_g} < -g \}) &\ge c e^{g r^*}.
 \end{align*}
 
Write $\zeta_n$ for the steps $u_{n} - u_{n-1}$; $\zeta_n = -1$ or $\zeta_n = 1+\alpha$ both with probability $1/2$.
Write 
 $S_n = \zeta_1 + \cdots + \zeta_n = u_n - u_0$ and consider the function 
\[
G_n = e^{r^* S_n},
\] 
where $r^* <0$ is the solution of
\[
\frac{1}{2}  e^{-r^*}  +  \frac{1}{2} e^{r^*(1+\alpha)} = 1.
\]
One can check that this equation has a unique solution $r^* <0$ with $r^* \to 0$ as $\alpha \to 0$.  
Now $G_n$ is a martingale as 
\begin{align*}
\int_{C^{0,\ldots,n-1}_{\omega_0,\ldots,\omega_{n-1}}}  & e^{r^* S_n}  \, d\nu^+ (\omega) = 
\int_{C^{0,\ldots,n-1}_{\omega_0,\ldots,\omega_{n-1}}}  e^{r^* S_{n-1}} e^{r^* \zeta_n}  \, d\nu^+ (\omega) 
\\
&= e^{r^* S_{n-1}}  \int_{C^{0,\ldots,n-1}_{\omega_0,\ldots,\omega_{n-1}}}  e^{r^* \zeta_n}  \, d\nu^+ (\omega) 
\\
&= e^{r^* S_{n-1}} \left(  \int_{C^{0,\ldots,n}_{\omega_0,\ldots,\omega_{n-1},1}}  e^{-r^*}  \, d\nu^+ (\omega) + 
\int_{C^{0,\ldots,n}_{\omega_0,\ldots,\omega_{n-1},2}}  e^{r^* (1 + \alpha)}  \, d\nu^+ (\omega) \right)
\\
&=  e^{r^* S_{n-1}}  \int_{C^{0,\ldots,n-1}_{\omega_0,\ldots,\omega_{n-1}}}  \frac{1}{2}  e^{-r^*}  +  \frac{1}{2} e^{r^*(1+\alpha)}\, d\nu^+ (\omega)
%
\\
&=  \int_{C^{0,\ldots,n-1}_{\omega_0,\ldots,\omega_{n-1}}}  e^{r^* S_{n-1}}  \, d\nu^+ (\omega). 
\end{align*}
By Doob's optional stopping theorem, 
see e.g. \cite[Theorem~VII.2.2]{shi84},
\begin{align*}
\int_{\Sigma_2^+} e^{r^* S_{U_g}} \, d\nu^+ (\omega) = \int_{\Sigma_2^+} e^{r^* S_{0}} \, d\nu^+ (\omega)  =  1.
\end{align*}

This gives
\begin{align*}
\int_{\Sigma_2^+} e^{r^* u_{U_g}} \, d\nu^+ (\omega)  =  e^{r^* u_0}.
\end{align*}
Observe  $u_{U_g} \in [-g-1,-g)$ 
or $u_{U_g} \in (0,1+\alpha]$. 
Let
\[A = 
\nu^+ (\{ \omega \in \Sigma_2^+ \;  \mid  \; u_{U_g} < -g \}) 
\]
be the probability that $u_{U_g} < -g$.
Write
\begin{align*}
\int_{\Sigma_2^+} e^{r^* u_{U_g}} \, d\nu^+ (\omega) &= 
 A  e^{-g r^*} e^{-c_1 r^*} +  (1-A)  e^{c_2 r^*},
\end{align*}
where
\begin{align*}
  e^{-g r^*} e^{-c_1 r^*} &= \frac{1}{A} \int_{ \{ \omega \in \Sigma_2^+ \; \mid \; u_{U_g} < -g\}} e^{r^* u_{U_g}} \, d\nu^+ (\omega),
\\
 e^{c_2 r^*} &= \frac{1}{1-A} \int_{ \{ \omega \in \Sigma_2^+ \; \mid \; u_{U_g} > 0\}}e^{r^* u_{U_g}} \, d\nu^+ (\omega).
\end{align*}
In these expressions,  $0\le c_1 \le 1$, $0 \le c_2 \le 1+\alpha$.
A moment of thought gives that $c_2 >0$ 
($c_2 =0$ can only occur if points leave $[-g,0]$ through the right boundary point $0$, but the initial point $u_0 \in (-1,0]$ is mapped with probability $1/2$ to a point in $(\alpha,1+\alpha]$).
%
We obtain
\begin{align*}
A  \left( e^{-g r^*} e^{-c_1 r^*} - e^{c_2 r^*} \right) &= e^{u_0 r^*} - e^{c_2 r^*},
\end{align*}
where this last number is positive.
\end{proof}

Similar arguments 
prove the following lemma.

\begin{lemma} \label{l:Tfinite}
\begin{align}\label{e:waldsigma}
 \int_{\Sigma_2^+}  \xi_k (\omega) \, d \nu^+ (\omega) < \infty;
\end{align}
$\xi_k$ has finite expectation.
\end{lemma}

\begin{proof}
Recall that $x_{n+1} = f_{\omega_n} (x_n)$ on $\mathbb{I}$ is conjugate to $y_{n+1} = h_{\omega_n} (y_n)$ on $\mathbb{R}$
through  $y_n = \ln (x_n / (1-x_n))$.
We split iterates of $y_n$ in $[K,\infty)$ into two sets, namely iterates in $[K,\tilde{K}]$ and iterates in $(\tilde{K} , \infty)$,
for some positive and large $\tilde{K}$.
 Near $x=1$, write $f_i (x) = 1- a_i (1-x) (1 + r_i(1-x))$ with $a_i >0$ and $r_i(u) = O(u)$, $u \to 0$.
 The positive Lyapunov condition $L(1) > 0$ means that  
$\ln(a_1) + \ln(a_2) > 0$.
 Calculate
 \begin{align*}
 y_{n+1} &= y_n - \ln (a_{\omega_n}) + 
 \\
 & \ln  ( 1 - (1 - x_n)) + \ln (1 + r_{\omega_n} (1-x_n)) + \ln (1 - a_{\omega_n} (1 - x_n)(1 + r_{\omega_n} (1-x_n))),
 \end{align*}
 where $1 - x_n = 1 / (1 + e^{y_n})$.
 From this expression it is easily seen for any $\varepsilon>0$ one can pick $\tilde{K}$ large,
 so that for $y_n> \tilde{K}$, \[y_{n+1} \le y_n - \ln (a_{\omega_n}) + \varepsilon.\] 
 Pick $\varepsilon$ small enough so that 
 $- \ln (a_1)  - \ln (a_2) + 2 \varepsilon < 0$.
 
 For $z_0 \in (\tilde{K},h_2 (\tilde{K})]$ and $z_{n+1} = h_{\omega_n} (z_n)$, let 
 \[T_{\tilde{K}} = \min\{ n \in \mathbb{N} \; \mid \; z_n \le \tilde{K}\}\]
 be the stopping time to leave  $(\tilde{K},\infty)$.
 As in the proof of Lemma~\ref{l:Tinfty} one shows that the expectation of  $T_{\tilde{K}}$  is finite.
 To provide the argument, consider the random walk
 \[u_{n+1} = u_n - \ln (a_{\omega_n}) + \varepsilon\] 
 starting at $u_0 = z_0$ and let 
 \[T_u = \min \{n \in \mathbb{N} \; \mid \; u_n \le \tilde{K}\}.\]
 Then $T_{\tilde{K}} \le T_u$. By Wald's identity, $\int_{\Sigma_2^+} T_u (\omega) \, d\nu^+ (\omega) < \infty$ and hence
  \begin{align}
  \label{e:stopy}
  \int_{\Sigma_2^+} T_{\tilde{K}} (\omega) \, d\nu^+ (\omega) < \infty.
 \end{align}

After these preparations we define the first return map $g_\omega : (-\infty , \tilde{K}] \to (-\infty, \tilde{K}]$,
 %
 \[ 
 g_\omega (y) = h^{R(\omega,y)}_\omega (y),
 \]
 where \[R (\omega,y)  = \min \{ n \ge 1 \; \mid \; h^n_\omega (y) \le  \tilde{K} \}.\]
 By \eqref{e:stopy}, $R$ has finite expectation. In fact, there is $C>0$ so that for each $y \in [K,\tilde{K}]$,
   \begin{align}\label{e:stopR}
   \int_{\Sigma_2^+}  R(\omega,y) \, d\nu^+ (\omega) \le C.
   \end{align}
   
 The next step is to show that
 \[
 T_K = \min \{ n\in \mathbb{N} \; \mid \; g^n_\omega (y) < K \}
 \] 
 for 
 $y \in [K,\tilde{K}]$ has finite expectation.
 Consider the 
 skew product $G: \Sigma_2^+ \times (-\infty,\tilde{K}] \to \Sigma_2^+ \times (-\infty,\tilde{K}]$,
 \[
 G (\omega,y) = ( \sigma^{R(\omega,y)}  \omega , h^{R(\omega,y)}_\omega (y) ) =  ( \sigma^{R(\omega,y)}  \omega , g_\omega (y) ).
 \]
 Let $\pi$ be the projection $\pi (\omega,y) = \omega_0$ from $\Sigma_2^+ \times \mathbb{R}$ onto $\{1,2\}$.
 Given $(\omega,y) \in \Sigma_2^+ \times [K,\tilde{K}]$ we obtain a sequence $\rho \in \Sigma_2^+$ given by 
 \[\rho_i = \pi  G^i (\omega,y).\]
  It follows from the construction that as the sequence $(\omega_i)_0^\infty$ is independent and identically distributed, also
 $(\rho_i)_0^\infty$ is independent and identically distributed with the same distribution:
 probability $1/2$ for both symbols $1,2$.
Because $f_1 (x) < x$, we find $h_1 (y) < y$ and thus that there is a number $l < 0$ with
  \begin{align*}
  h_1 (y) < y + l,
  \end{align*}
  for $y \in [K, \tilde{K}]$.
  Hence, for any $y \in [K,\tilde{K}]$ and $N =\ceil{  (K - \tilde{K}) /l}$ we will have $g_1^N (y) = h_1^N (y) < K$.  
  The stopping time $T_K$ is therefore smaller than the stopping time
  \[
    \min \{ n \in \mathbb{N} \; \mid \;  \rho_i = 1 \text{ for } n-N < i \le n \}.
    \]
  Note that the expected number of throws of symbols $1,2$ that lead to $N$ consecutive $1$'s is finite.
  (In fact it equals $2^{N+1} -2$. It is easily bounded by $N$ times the expectation of the first number $j$ so that $\omega_i = 1$ for $j N \le i < j (N+1)$;
  the latter is a geometric distribution with expectation $2^N$). 
 So
  the expectation of the stopping time $T_K$ is finite;
  \begin{align}\label{e:stopK}
  \int_{\Sigma_2^+} T_{K}(\omega) \, d\nu^+(\omega) < \infty.
  \end{align}
   
 Finally we combine \eqref{e:stopR} and \eqref{e:stopK}: the formula 
\begin{align*}
 \xi_k (\omega) = \sum_{n=0}^{T_K (\omega)-1}  R (G^n (\omega, y_{T_{2k-1}}))
\end{align*} 
implies that 
 \begin{align*}
  \int_{\Sigma_2^+} \xi_k (\omega) \, d\nu^+ (\omega)  &\le C \int_{\Sigma_2^+} T_K (\omega) \,d\nu^+(\omega) 
    \\
     &< \infty.
  \end{align*}
 This proves Lemma~\ref{l:Tfinite}.
\end{proof}

We can now finish the proof of Theorem~\ref{t:delta0}.
Define for $n \in [T_{2k},T_{2k+1})$,
\begin{align*}
 N_\eta (n) = k, &\;
 N_\xi (n) = k
\end{align*}
and
\begin{align*}
 \tilde{\eta} (n) =  n +1 - T_{2k}, &\; 
 \tilde{\xi} (n) = 0,
\end{align*}
so that $\tilde{\eta}$ counts the number of iterates from $T_{2k}$ on where $y_n \le K$.
Likewise define for $n \in [T_{2k+1},T_{2k+2})$,
\begin{align*}
 N_\eta (n) = k+1, &\;
 N_\xi (n) = k
\end{align*}
and
\begin{align*}
 \tilde{\eta}(n) = 0, &\;
 \tilde{\xi} (n) =  n +1 - T_{2k+1}. 
\end{align*}
So $\tilde{\xi}$ counts the number of iterates from $T_{2k+1}$ on where $y_n > K$.

Finally calculate 
\begin{align*}
& \frac{1}{n} \sum_{i=0}^{n-1}   \mathbbm{1}_{[0,\beta)} (f^i_\omega (x_0)) = 
            \frac{1}{n} \left( \sum_{k=1}^{N_\eta (n-1)} \eta_k + \tilde{\eta} (n-1) \right) 
\\
 &= \left( \sum_{k=1}^{N_\eta(n-1)} \eta_k + \tilde{\eta}(n-1) \right) \left/ \left( \sum_{k=1}^{N_\eta(n-1)} \eta_k + \tilde{\eta}(n-1) + \sum_{k=1}^{N_\xi (n-1)} \xi_k + \tilde{\xi} (n-1) \right) \right. 
  \\
  &= \left(  1 + \left( \sum_{k=1}^{N_\xi(n-1)} \xi_k + \tilde{\xi} (n-1)   \right) \left/ \left( \sum_{k=1}^{N_\eta(n-1)} \eta_k + \tilde{\eta}(n-1) \right)\right.    \right)^{-1}
 \\
&\ge \left( 1 + \left( \sum_{k=1}^{N_\xi (n-1) + 1} \xi_k \right) \left/ \left( \sum_{k=1}^{N_\eta(n-1)} \eta_k \right)\right.\right)^{-1}.
\end{align*}
By \eqref{e:waldrho} and \eqref{e:waldsigma},
the last term goes to $1$ for $\nu^+$-almost all $\omega$, as $n \to \infty$ (note that $N_\eta (n-1) - N_\xi(n-1) \le 1$). 
 \end{proof}

The next theorem  is an immediate consequence of Theorem~\ref{t:delta0}.

\begin{theorem}\label{t:stat01}
Let ${F^+} \in \mathcal{S}$ and assume $L(0)\le 0$ and $L(1) > 0$.
Then the only ergodic stationary measures are the delta measures at $0$ and $1$.
\end{theorem}

\begin{proof}
We will only treat the case $L(0)=0$ and $L(1) >0$.
Suppose there  is an ergodic stationary measure $m$ with support in $(0,1)$. 
By Lemma~\ref{l:corres}, $\nu^+ \times m$ is an ergodic invariant measure for $F^+$.
 By Birkhoff's ergodic theorem, for $\nu^+\times m$-almost every $(\omega,x)$, we have
 \begin{equation}\label{e:distrre}
 \lim_{n\to\infty} \frac{1}{n} \sum_{i=0}^{n-1} \delta_{f^i_{\omega} (x)} = m.
 \end{equation}
 By Fubini's theorem, there is a subset of $\mathbb{I}$ of full $m$-measure, so that 
 in any $\Sigma_2^+ \times \{ x\}$ with $x$ from this subset,
 there is a set of full $\nu^+$-measure for which \eqref{e:distrre} holds.
 This however contradicts \eqref{e:d0}, since that holds for all $\beta >0$ and applies to all $x\in \mathbb{I}$.
 %
\end{proof}

The type of reasoning to prove Theorem~\ref{t:delta0} can be used to obtain the following result 
on iterated functions systems with zero Lyapunov exponents at both end points.  

\begin{theorem}\label{t:delta01}
Consider ${F^+} \in \mathcal{S}$ and assume $L(0) = L(1) = 0$.
 Let $0< \beta < 1$ and $x_0 \in (0,1)$.
 Then for $\nu^+$-almost every $\omega \in \Sigma_2^+$,
 \begin{align*}
  \lim_{n\to \infty} \frac{1}{n} \sum_{i=0}^{n-1} \mathbbm{1}_{[\beta,1-\beta]} (f^i_\omega (x_0)) &= 0.
 \end{align*}
\end{theorem}

Figure~\ref{f:walk} illustrates a time series of the symmetric random walk, to which this theorem applies.

\subsection{Pullback convergence}\label{ss:pullback}

Theorem~\ref{t:beta} implies that
forward convergence of $f^n_\omega(x)$ to $0$ does not hold: it is not true that
for $\nu^+$-almost all $\omega \in \Sigma_2^+$,  
$f^n_{\omega}  (x)  \to 0$ as $n\to \infty$.
The next result stipulates that pullback convergence to $0$ does hold.
See also \cite[Section~9.3.4]{arn98} for a related example
where pullback convergence does not imply forward convergence,
in a context of stochastic differential equations. 

\begin{theorem}\label{t:pullback}
Let $F^+ \in \mathcal{S}$ and suppose $L(0) =0$ and $L(1) >0$.
Take $x \in (0,1)$.
Then for  $\nu$-almost all $\omega \in \Sigma_2$, 
\begin{align}\label{e:payback}
\lim_{n\to\infty} f^n_{\sigma^{-n} \omega}  (x)  &= 0.
\end{align}
\end{theorem}

\begin{proof} 
We reformulate the theorem to the following equivalent statement:
for $\nu$-almost all $\omega \in \Sigma_2$, and for all $y \in (0,1)$, 
\begin{align}\label{e:unstisall}
\lim_{n\to \infty} f^{-n}_\omega (y) &= 1.
\end{align}
Equivalence of the statements \eqref{e:payback} and \eqref{e:unstisall} follows from the monotonicity of the interval diffeomorphisms:
 $f^{-n}_\omega (y) > x$ precisely if $f^{n}_{\sigma^{-n} \omega} (x) < y$ and thus for $\varepsilon_1,\varepsilon_2$ small positive numbers,
 $f^{-n}_\omega (\varepsilon_1) > 1-\varepsilon_2$ precisely if $f^{n}_{\sigma^{-n} \omega} (1-\varepsilon_2) < \varepsilon_1$.
 
To prove \eqref{e:unstisall}, consider
\begin{align}\label{e:uis}
u(\omega) &= \inf \{ y \in \mathbb{I} \; \mid \; \lim_{i\to\infty} f^{-i}_\omega (y) = 1 \}. 
\end{align}
As $L(1) >0$, 
by Lemma~\ref{l:bonmil} we know that $u$ exists and $u<1$, $\nu$-almost everywhere.
Since $u$ is invariant we get that either $u>0$, $\nu$-almost everywhere, or $u=0$, $\nu$-almost everywhere.
Assume that $u$ is not identically $0$. 
The measure $\mu = (\text{id},u)\nu$ on $\Sigma_2 \times \mathbb{I}$ with conditional measures $\delta_{u(\omega)}$ and marginal $\nu$ on $\Sigma_2$,
defines an invariant measure for $F$. 

Denote by $\Pi$ the natural projection $\Sigma_2 \times \mathbb{I} \to  \Sigma_2^+ \times \mathbb{I}$,
where $\Sigma_2 = \Sigma_2^- \times \Sigma_2^+$.
Expression \eqref{e:uis} gives that $u(\omega)$ depends on the past $\omega^-=  (\omega_i)_{-\infty}^{-1}$ only. 
Therefore, the measure $\mu$ is a product measure $\nu^+ \times \vartheta$ on
$\Sigma_2^+ \times \left(\Sigma_2^- \times \mathbb{I}\right)$.
The projection $\Pi \mu$ is therefore a product measure  $\nu^+ \times m$ on
$\Sigma_2^+ \times \mathbb{I}$. 
That is, $\mu$ corresponds to an invariant measure $\nu^+ \times m$ for $F^+$, see Proposition~\ref{p:fapp}.
Here $m$ is a stationary measure by Lemma~\ref{l:corres}.
Since $0<u<1$, $\nu$-almost everywhere,  $m$ assigns positive measure to $(0,1)$.

By Theorem~\ref{t:stat01}, the only stationary measures are convex combinations of delta measures at $0$ and $1$.
We have obtained a contradiction and proven \eqref{e:unstisall} and hence the theorem.
\end{proof}

\subsection{Central limit theorem} 

Under the assumptions of Theorem~\ref{t:pullback}, its  conclusion that $f^n_{\sigma^{-n} \omega} (x) \to 0$ for $\nu$-almost all $\omega$, 
implies that $f^n_{\sigma^{-n} \omega} (x)$ converges to $0$ in probability.
By $\sigma$-invariance of $\nu$,  $f^n_\omega (x)$ converges to $0$ in probability. Hence, for any $a \in (0,1)$,
\begin{align*} 
\lim_{n\to \infty} \nu^+ (\{ \omega \in \Sigma_2^+ \;  \mid  \;  f^n_\omega (x) \le a    \}) &= 1.
\end{align*}
We state a central limit theorem that gives convergence of the distribution of the points $f^n_\omega(x)$, after an appropriate scaling.
The proof is essentially contained in
\cite{lam62}, where a central limit theorem for Markov processes on the half-line is stated.

\begin{theorem}
Let $F^+ \in \mathcal{S}$ and assume $L(0) = 0$ and $L(1) > 0$. Let $x \in (0,1)$.
Then, for $a >0$, 
\begin{align*}
\lim_{n\to \infty} \nu^+ \left(\left\{ \omega \in \Sigma_2^+  \;  \mid  \; f^n_\omega (x) \ge e^{-a \sqrt{n}} \right\} \right)= \int_0^a \frac{2 e^{-\xi^2/2}}{\sqrt{2\pi}} \, d\xi.
\end{align*}
\end{theorem}

\begin{proof}
Take $x_{n+1} = f_{\omega_n} (x_n)$ with $x_0 = x$.
Write $y_n = - \ln (x_n) + \ln (x)$, so that $y_0 = 0$ and $y_n \in [0,\infty)$ if $x_n \in (0,x]$.
Write $z_n = \max\{ 0 , y_n\}$.

\begin{lemma}\label{l:evenmoments}
The moments of  $z_n^2/n$ satisfy
\[ \lim_{n\to \infty} \int_{\Sigma_2^+} z_n^{2k} / n^{k} \, d\nu^+= 
\frac{(2k)!}{2^k k!}.\]
\end{lemma}

\begin{proof}
It suffices to follow the proof of \cite[Lemma~2.1]{lam62} and \cite[Lemma~2.2]{lam62}. 
The proofs in \cite{lam62} use that the process is null, that is, for any compact interval $I \subset \mathbb{R}$, 
$\lim_{n\to \infty} \frac{1}{n} \sum_{i=0}^{n-1} \nu^+ (\{ \omega \;  \mid  \; z_i\in I \}) = 0$. 
 This holds by Theorem~\ref{t:delta0}.
\end{proof}

As in \cite[Theorem~2.1]{lam62}, 
Lemma~\ref{l:evenmoments} implies that $z_n^2/n$ has a limiting distribution
as $n\to \infty$ and
\[
\lim_{n\to \infty} \nu^+ \left( \left\{  \omega\in\Sigma_2^+ \;  \mid  \;  \frac{z_n^2}{n}     \le  a^2    \right\}\right)
= \int_0^{a}    \frac{2 e^{-\xi^2/2}}{\sqrt{2\pi}} \, d\xi.
\] 
We conclude that
\[
\lim_{n\to \infty} \nu^+ \left( \left\{  \omega\in\Sigma_2^+ \;  \mid  \;  \frac{y_n}{\sqrt{n}}     \le  a    \right\}\right)
= \int_0^{a}    \frac{2 e^{-\xi^2/2}}{\sqrt{2\pi}} \, d\xi.
\] 
Plugging in $y_n = -\ln (x_n) + \ln(x)$ gives the statement of the theorem.
\end{proof}


%
%

\section{Random walk with drift}\label{s:riddled}

The material in the previous sections treats all possible combinations of signs of $L(0)$ and $L(1)$ 
except the case where $L(0) \ge 0$ and $L(1) < 0$ (or vice versa). We put the remaining case in the following result.

\begin{theorem}\label{t:forward}
Consider ${F^+} \in \mathcal{S}$ and assume $L(0) \ge 0$  and $L(1) < 0$.
Let $x_0 \in (0,1)$.  
Then for $\nu^+$-almost every $\omega \in \Sigma_2^+$,
 \begin{align*}
  \lim_{n\to \infty} f^n_\omega (x_0) &= 1.
 \end{align*}
\end{theorem}

\begin{proof}
Take the proof of Theorem~\ref{t:pullback} applied to the inverse skew product system $F^{-1}$.
\end{proof}

Theorem~\ref{t:forward} establishes forward convergence of $f^n_\omega (x_0)$ to $1$ under the given assumptions.
Consider ${F^+} \in \mathcal{S}$ with $L(0) > 0$  and $L(1) < 0$.
Then there is also pullback convergence to $1$;
\begin{align*}
  \lim_{n\to \infty} f^n_{\sigma^{-n} \omega} (x_0) &= 1
 \end{align*}
for $\nu$-almost all $\omega \in \Sigma_2$.
It follows from the results in Section~\ref{s:on-off}, again by going to the inverse skew product system, 
that such a pullback convergence does not hold in case
$L(0)=0$, $L(1)<0$. See \cite{cra02} for considerations on forward versus pullback convergence in
 an example of random circle dynamics.

\appendix

\section{Invariant measures for step skew product systems}\label{s:invmeasures}


An iterated function system defines a Markov process and as such may admit stationary measures.
Their relation with invariant measures for the corresponding one-sided skew product system and its natural extension, 
the two-sided skew product system, is explored in this section. 
This is classical material, originating
from Furstenberg \cite{fur73}. A general account of the constructions is found in \cite{arn98}.
We provide a simplified discussion taylored to a setting of step skew product systems over shifts.
The reader may also consult the exposition in \cite[Chapter 5]{via14}.

Assume the context from Section~\ref{s:sps}.
So consider $\Omega = \{1,\ldots,N\}$ and the family of diffeomorphisms $\mathbb{F}=\{f_1,\ldots,f_N\}$ on $M$.
We pick $f_i$ with probability $p_i$ with $0<p_i<1$ and $\sum_{i=1}^{N} p_i=1$.

We endow $\Sigma_N$ with the Borel sigma-algebra, denoted by $\mathcal{F}$.
Likewise we take the Borel sigma-algebra $\mathcal{F}^+$ on $\Sigma_N^+$.  
Given the probabilities $p_i$, we take a Bernoulli measure $\nu$ on $\Sigma_{N}$ 
which is determined by its values on cylinders;
\[
\nu (C^{k_1,\ldots,k_n}_{\omega_1,\ldots,\omega_n}) = \prod_{i=1}^n  p_{\omega_i}.
\]
Write $\nu^+$ for the Bernoulli measure on $\Sigma_N^+$, defined analogously.
For different probability vectors $(p_1,\ldots,p_N)$, 
the corresponding Bernoulli measures are mutually singular.

Denote by $\mathcal B$ the Borel sigma-algebra on $M$.
For a measure $m$ on $M$ and any $\mathcal B$-measurable set $A$ we denote the push-forward measure of $m$ 
by $f_i m$ in which \[ f_i m_i(A)= m_i(f_i^{-1}(A)).\] 

\begin{definition}
A 
stationary measure $m$ on $M$ is a probability measure that equals its average push-forward under the iterated function 
system 
$\text{IFS}\,(\mathbb{F})$, i.e. it satisfies \[ m = \sum_{i=1}^{N} p_i f_i m.\]
\end{definition}

Write $\mathcal{P}_M$ for the space of probability measures on $M$, equipped with the weak star topology.
Write $\mathcal{T}: \mathcal{P}_M \to \mathcal{P}_M$ for the map
\[
\mathcal{T} m = \sum_{i=1}^{N} p_i f_i m.
\]
A stationary measure is a fixed point of $\mathcal{T}$.

\begin{lemma}\label{l:c0}
The map $\mathcal{T}$ is continuous.
It also depends continuously on the parameters $p_1, \ldots, p_N$ and $f_1,\ldots,f_{N}$.
\end{lemma}

\begin{proof}
Recall that a sequence of measures $m_n$ converges in the weak star topology to a measure $m$ precisely if for each open set $A$,
$\liminf_{n\to \infty} m_n (A) \geqslant m (A)$, see e.g. \cite[Theorem~III.1.1]{shi84}. 

If $m_n$ converges to $m$ it follows that for $A$ open,
\[
\liminf_{n\to \infty} f_i m_n (A)  = \liminf_{n\to\infty} m_n (f_i^{-1}(A)) \geqslant m (f_i^{-1}(A)) 
\]
 since $f_i^{-1} (A)$ is open.
That is, 
\[
\liminf_{n\to \infty} f_i  m_n (A)  \geqslant f_i m (A)
\]
and thus $f_i  m_n$ converges to $f_i  m$.
This argument also shows that $\mathcal{T}$ is continuous.

To prove continuous dependence on $f_1,\ldots,f_{N}$, consider a sequence of maps $f_{i,n}$ converging to $f_i$. 
By inner regularity, for an open set $O\subset M$ one has 
$m (O) = \sup_{C \subset O} m(C)$, where $C$ runs over compact subsets of $O$. 
So also, 
given $\varepsilon>0$, for $A \subset M$ open, there exists compact $K \subset A$ 
with $m (f_i^{-1} (K)) \geqslant m(f_i^{-1} (A)) -\varepsilon$.

Then, for $A$ open, given $\varepsilon>0$ there are $n_0 >0$ and $K \subset A$ so that 
$f_{i}^{-1} (K) \subset f_{i,n}^{-1} (A)$ for $n \geqslant n_0$ and $m (f_i^{-1} (K)) \geqslant m(f_i^{-1}(A)) - \varepsilon$.   
So,
\begin{align*}
\liminf_{n\to \infty} f_{i,n}  m (A) &= \liminf_{n\to\infty} m (f_{i,n}^{-1} (A))
\\
&\ge \lim_{n_0\to \infty} m  \left( \bigcap_{n \ge n_0} f_{i,n}^{-1} (A) \right) 
\\ 
&\ge m (f_i^{-1} (K)) 
\\
&\ge m (f_i^{-1} (A )) - \varepsilon.
\end{align*}
As this holds for any $\varepsilon$, we get
\[ 
\liminf_{n\to \infty} f_{i,n} m (A) \geqslant m (f_i^{-1} (A ))
\] 
and hence that $f_{i,n} m$ converges to $f_i m$.
This argument shows that $\mathcal{T}$  depends continuously on $f_1,\ldots,f_N$, continuous dependence on
parameters $p_1,\ldots,p_N$ is clear.
\end{proof}

The same type of argument shows that the map $\mathcal{T}_\varepsilon$, appearing in the proof of Theorem~\ref{t:sync},
is continuous and changes continuously with $\varepsilon$. 
The set of fixed points of $\mathcal{T}$ changes upper semi-continuously in the Hausdorff metric if
parameters $p_1,\ldots,p_N$ and $f_1,\ldots,f_N$ are varied.
So if $m$ is a unique fixed point for $\mathcal{T}$, $\mathcal{T}_\varepsilon \to \mathcal{T}$ as $\varepsilon \to 0$ and
$\mathcal{T}_\varepsilon m_\varepsilon = m_\varepsilon$, then $m_\varepsilon \to m$ as $\varepsilon \to 0$.

\begin{lemma}\label{l:corres}
A probability measure $m$ is a stationary measure if and only if
  $\mu^+ = \nu^+ \times m$
is an invariant measure of $F^+$ with marginal $\nu^+$ on $\Sigma_N^+$. 
\end{lemma}

\begin{proof}
Consider the following calculation
for product sets $C\times B \subset \Sigma_N^+ \times M$ of a cylinder $C = C^{0,\ldots,n-1}_{i_0,\ldots,i_{n-1}}$ and a Borel set $B$:
\begin{align*}
F^+ (\nu^+ \times m) (C \times B) &= \nu^+ \times m ((F^+)^{-1} (C\times B)) 
\\
&= \sum_{i=1}^N \nu^+ \times m  \left( C^{0,1,\ldots,n}_{i,i_0,\ldots,i_{n-1}} \times f_i^{-1}(B) \right)
\\
&= \sum_{i=1}^N p_i \nu^+ (C) m (f_i^{-1} (B))
\\
&=   \sum_{i=1}^N p_i \nu^+ (C) f_i m (B).
\end{align*}
If $m$ is a stationary measure, then the last expression equals $\nu^+ (C) m(B) = \nu^+ \times m (C \times B)$, so that
$F^+  (\nu^+ \times m) (C \times B) = \nu^+ \times m (C \times B)$.
Since the product sets generate  the $\sigma$-algebra, this 
proves $F^+$-invariance of $\nu^+ \times m$.
Similarly, if $\nu^+\times m$ is $F^+$-invariant, then 
the last expression equals $\nu^+ \times m (C \times B) = \nu^+ (C) m(B)$
and 
this proves $\sum_{i=1}^N p_i f_i  m (B) =  m (B)$. 
\end{proof}

Let $m$ be a stationary measure for $M$. 
We say that $m$ is ergodic if $\nu^+ \times m$ is ergodic for $F^+$. 
A point $(\omega,x)$ is said to be a generic point for an ergodic measure $\nu^+ \times m$,
if the orbit is distributed according to the measure.


Write $\pi: \Sigma_2 \to \Sigma_2^+$ for the natural projection $(\omega_n)_{-\infty}^\infty \mapsto (\omega_n)_0^\infty$.
The Borel sigma-algebra $\mathcal{F}^+$ on $\Sigma_N^+$  
yields a sigma-algebra $\mathcal{F}_0 = \pi^{-1}\mathcal{F}^+$ on $\Sigma_{N}$.
%
A  measure  $\mu$ on $\Sigma_{N} \times M$ with marginal $\nu$ has conditional measures $\mu_\omega$ 
on 
the fibers $\{\omega\}\times M$, such that
\begin{align}\label{e:disint}
\mu(A) &= \int_{\Sigma_{N}} \mu_\omega (A_\omega) \, d \nu(\omega)
\end{align}
for measurable sets $A$, where we have written 
\[
A_\omega = A \cap (\{\omega\} \times M).
\]
A measure $\mu^+$ on $\Sigma_{N}^+ \times M$ with marginal $\nu^+$ likewise has conditional measures $\mu^+_\omega$.
It 
is convenient to consider $\nu^+$ as a measure on $\Sigma_{N}$ with sigma-algebra $\mathcal{F}_0$ 
and
$\mu^+$ as a measure on $\Sigma_{N} \times M$ with sigma-algebra $\mathcal{F}_0 \otimes \mathcal{B}$.
When 
$\omega \in \Sigma_{N}$ we will write $\mu^+_\omega$ for the conditional measures $\mu^+_{ \pi \omega}$.
The spaces 
of measures are equipped with the weak star topology.

Invariant measures for $F^+$ with marginal $\nu^+$ correspond to 
invariant measures for $F$ with marginal $\nu$  
in a one-to-one relationship, as detailed in Proposition~\ref{p:fapp} below. 
This is a special case of \cite[Theorem~1.7.2]{arn98}.
The result implies that stationary measures correspond one-to-one to specific
invariant measures for $F$ with marginal $\nu$.

Write $\Sigma_N = \Sigma_N^- \times \Sigma_N^+$, where $\Sigma_N^-$ consists of the past parts $(\omega_i)_{-\infty}^{-1}$
of sequences $\omega$.
We have a natural projection 
\[\Pi : \Sigma_N^- \times \Sigma_N ^+ \times M \to \Sigma_N^+ \times M.\]

\begin{proposition}\label{p:fapp}
Let 
$\mu^+$ be an $F^+$-invariant probability measure with marginal $\nu^+$.
Then,
there exists an $F$-invariant probability measure $\mu$ with marginal $\nu$ and conditional measures
\begin{align}\label{e:muapp}
\mu_\omega &=  \lim_{n\to \infty}  f^n_{\sigma^{-n} \omega} \mu^+_{\sigma^{-n} \omega},
\end{align}
$\nu$-almost surely.

Let 
$\mu$ be an $F$-invariant probability measure with marginal $\nu$.
Then,
\begin{align}\label{e:mu+app}
\mu^+ &= 
\Pi \mu
\end{align}
is an $F^+$-invariant probability measure with marginal $\nu^+$.

The correspondence $\mu \leftrightarrow \mu^+$ given by \eqref{e:muapp}, \eqref{e:mu+app} is one-to-one.
Furthermore, through these relations, $F^+$-invariant product measures $\mu^+ = \nu^+ \times m$ correspond
one-to-one with $F$-invariant product measures $\mu = \nu^+ \times \vartheta$ on $\Sigma_N^+ \times (\Sigma_N^- \times M)$. 
\end{proposition}

\begin{remark}\label{r:f}
Consider $\mu^+$ as a measure on $\Sigma_N \times M$ with sigma-algebra $\mathcal{F}_0 \otimes \mathcal{B}$.
Observe that $F^n (\mu^+)$ has conditional measures
$f^n_{\sigma^{-n} \omega} \mu^+_{\sigma^{-n} \omega}$ on $\{\omega\}\times M$.
Hence, \eqref{e:muapp} reads
\begin{align*} 
\mu &= \lim_{n\to\infty}  F^n (\mu^+).
\end{align*}
\end{remark}

\begin{remark}
From the characterization of ergodic probability measures as extremal points 
in
the set of invariant probability measures,
the 
one-to-one correspondence $\mu \leftrightarrow \mu^+$ implies that $\mu$ 
is 
ergodic if and only if $\mu^+$ is ergodic.
\end{remark}

\begin{proof}[Proof of Proposition~\ref{p:fapp}]
Note that 
$\mathcal{F}_s = \sigma^s \mathcal{F}_0$ are sigma-algebras on $\Sigma_N$ 
with
$\mathcal{F}_s  \uparrow\mathcal{F}$.
For a Borel set $B \subset M$, define
\[
 \upsilon_t (\omega) = f^t_{\sigma^{-t} \omega}  \mu^+_{\sigma^{-t}\omega} (B).
\]
Calculate, for $A_s = \sigma^{s} A_0 \in \mathcal{F}_s$ and $0 \le s \le t$,
\begin{align*}
\int_{A_s}  \upsilon_t (\omega)  \, d\nu (\omega) & \stackrel{(1)}{=}  
 \int_{A_s}  f^t_{\sigma^{-t} \omega} \mu^+_{\sigma^{-t}\omega} (B) \, d\nu(\omega)
 \\
 & \stackrel{(2)}{=}   \int_{A_0}  f^{t}_{\sigma^{s-t} \omega} \mu^+_{\sigma^{s-t}\omega} (B) \, d\nu(\omega)
 \\
 & \stackrel{(3)}{=}  \int_{A_0}  f^s_\omega \mu^+_{\omega} (B) \, d\nu(\omega)
 \\
 & \stackrel{(4)}{=}  \int_{A_s}  f^{s}_{\sigma^{-s} \omega} \mu^+_{\sigma^{-s}\omega} (B) \, d\nu(\omega)
 \\
 & \stackrel{(5)}{=}  \int_{A_s} \upsilon_s (\omega)  \, d\nu(\omega).
\end{align*}
Here (1) and (5) is the definition of $\upsilon_t$, (2) and (4) are by $\sigma$-invariance of $\nu$, and (3) is by $F^+$-invariance of $\mu^+$
(see Lemma~\ref{l:l} and Corollary~\ref{c:l} below for a derivation).

The above calculation shows that  $\upsilon_t $ is a martingale with respect to the filtration $\mathcal{F}_t$.
Hence the limit 
$\lim_{t\to \infty} \upsilon_t (\omega)$ exists.
By the Vitali-Hahn-Saks theorem, see \cite[Theorem~III.10]{do93}, the limit for varying Borel sets $B$ defines a measure, $\mu_\omega$. 
To obtain that the resulting measure $\mu$ is $F$-invariant, we refer to Remark~\ref{r:f}.
Since $F$ acts continuous on the space of probability measures, the limit $\lim_{n\to\infty}  F^n (\mu^+)$ is $F$-invariant.

It remains to show that $\mu$ and $\mu^+$ are in one-to-one correspondence.
We wish to show that, given $\mu$ and computing $\mu^+ = \Pi \mu$,
the formula \eqref{e:muapp} recovers $\mu$. 
Note, again with $A_t = \sigma^t A_0 \in \mathcal{F}_t$ for $0\le t$, 
%
\begin{align*}
\int_{A_t} & \upsilon_t(\omega) \, d \nu(\omega)  = \int_{A_t}  f^t_{\sigma^{-t}\omega} \mu^+_{\sigma^{-t} \omega} (B) \, d \nu(\omega)
\\ 
&= \int_{A_0} f^t_\omega \mu^+_\omega (B) \, d \nu(\omega) 
\\
&= \int_{A_0} \mu^+_\omega  ((f^t_\omega)^{-1} (B) ) \, d\nu(\omega)
\\ 
&=
\mu^+ \left(  \bigcup_{\omega \in A_0}  \{\omega\} \times (f^t_\omega)^{-1}(B)  \right)
\\
 &= \mu \left(  \bigcup_{\omega \in A_0}  \{\omega\} \times (f^t_\omega)^{-1}(B)  \right)
 \\ 
 &= \mu \left( (F^t)^{-1} ( A_t \times B ) \right)
 \\ 
 &= \int_{A_t} \mu_\omega (B)\, d \nu(\omega).
\end{align*}
As $\mathcal{F}_t \uparrow \mathcal{F}$, this shows that $\upsilon_t$ converges to $\mu$ as $t \to \infty$.

If $\mu$ is a product measure $\nu^+ \times \vartheta$ on $\Sigma^+_N \times (\Sigma_N^- \times M)$, 
then clearly  $\mu^+ = \Pi \mu$ is a product measure on $\Sigma_N^+ \times M$.
In the other direction, if $\mu^+ =  \nu^+ \times m$, then \eqref{e:muapp} reads
\begin{align*} 
\mu_\omega &=  \lim_{n\to \infty}  f^n_{\sigma^{-n} \omega} m,
\end{align*}
so that $\mu_\omega$ does not depend  on the future $\omega^+ = (\omega_n)_0^\infty$ of $\omega$.
For a product set $A = C^+ \times B \subset \Sigma_N^+ \times (\Sigma_N^- \times M)$, \eqref{e:disint} yields
\begin{align*}
\mu (A) &=  \int_{\Sigma_N^+}  \int_{\Sigma_N^-}   \mu_\omega (A_\omega) \, d\nu^- (\omega^-) d\nu^+ (\omega^+).
\end{align*}
Since $\mu_\omega$ depends on the past $\omega^-$ alone,
this can be written as $\mu (A) = \nu^+  \times \vartheta (C^+ \times B) = \nu^+ (C^+) \vartheta (B)$.
So $\mu$ is a product measure 
$\nu^+ \times \vartheta$ on $\Sigma^+_N \times (\Sigma_N^- \times M)$.
\end{proof}

The following lemma draws conclusions from $F^+$-invariance of $\mu^+$.

\begin{lemma}\label{l:l}
 For $A_0 \in \mathcal{F}^+$, 
 $B \in \mathcal{B}$, 
$0\le s \le t$,
 \[ 
 \int_{\sigma^{s-t} A_0} f^t_\omega \mu^+_\omega (B) \, d\nu^+ (\omega)  
 = \int_{\sigma^{s-t} A_0}  f^s_{\sigma^{t-s}\omega}  \mu^+_{\sigma^{t-s}\omega} (B) \, d \nu^+(\omega).
 \]
\end{lemma}

\begin{proof}
Write $A_{s-t} = \sigma^{s-t} A_0$ and compute, using $F^+$-invariance of $\mu^+$,
\begin{align*}  
\int_{A_{s-t}} & f^t_\omega \mu^+_\omega (B) \, d\nu^+ (\omega) =  
           \int_{A_{s-t}} f^s_{\sigma^{t-s} \omega} f^{t-s}_{ \omega} \mu^+_\omega (B) \, d\nu^+(\omega) 
\\
&=          \int_{A_{s-t}} \mu^+_\omega  (   (f^{t-s}_{\omega})^{-1}  (f^s_{\sigma^{t-s} \omega})^{-1} (B)) \, d\nu^+(\omega)
\\
& =     \mu^+  \left( \bigcup_{\omega \in A_{s-t}}  \{\omega\} \times     (f^{t-s}_{\omega})^{-1}  (f^s_{\sigma^{t-s} \omega})^{-1} (B) \right)
\\
& = 
 \mu^+  \left( \bigcup_{\omega \in A_{0}}  \{\sigma^{s-t} \omega\} \times     (f^{t-s}_{\sigma^{s-t}\omega})^{-1}  (f^s_{\omega})^{-1} (B) \right)
\\
& = \mu^+ \left(  (F^+)^{s-t}    \bigcup_{\omega \in A_{0}}  \{\omega\} \times   (f^s_{\omega})^{-1} (B) \right) 
\\
&= \mu^+ \left(   \bigcup_{\omega \in A_{0}}  \{\omega\} \times   (f^s_{\omega})^{-1} (B) \right) 
\\
& = \int_{A_{0}} \mu^+_\omega  (  (f^s_{\omega})^{-1} (B)) \, d\nu^+(\omega)
\\
&= \int_{A_{0}} f^s_{\omega} \mu^+_\omega (B) \, d\nu^+(\omega).
\end{align*}
As $\int_{A_{0}} f^s_{\omega} \mu^+_\omega (B) \, d\nu^+(\omega) = 
\int_{A_{s-t}} f^s_{ \sigma^{t-s} \omega} \mu^+_{\sigma^{t-s} \omega} (B) \, d\nu^+(\omega)$ 
by $\sigma$-invariance of $\nu^+$, this concludes the argument.
\end{proof}

\begin{corollary}\label{c:l}
The lemma implies that for $A_0 \in \mathcal{F}_0$, $B \in \mathcal{B}$, 
for $0\le s\le t$, 
\[
\int_{A_0}  f^{t}_{\sigma^{s-t} \omega} \mu^+_{\sigma^{s-t}\omega} (B) \, d\nu(\omega)
=  \int_{A_0}  f^s_\omega \mu^+_{\omega} (B) \, d\nu(\omega).
\]
\end{corollary}

Note that for the natural extension, $F$-invariance of $\mu$ means
\[
f^{t}_{\sigma^{s-t} \omega} \mu_{\sigma^{s-t}\omega}
=  f^s_\omega \mu_{\omega} \]
for $0\le s \le t$ and for $\nu$-almost all $\omega\in \Sigma_N$.

\def\cprime{$'$}

\end{document}